\documentclass{amsproc}

\usepackage{url}
\usepackage{supertabular} 

\usepackage{mathabx}
\usepackage{multirow} % for tables
\usepackage{graphics}

\allowdisplaybreaks 

\newcommand\scalemath[2]{\scalebox{#1}{\mbox{\ensuremath{\displaystyle #2}}}} %% To make the equations smaller, in Table~\ref{1to35}

\newcommand{\ud}{\mathrm{d}} 
\newcommand{\uD}{\mathrm{D}} 
\newcommand{\oldphi}{\phi}
\renewcommand{\phi}{\varphi}

\newtheorem{theorem}{Theorem}[section]

\theoremstyle{definition}

\theoremstyle{remark}

\theoremstyle{remark}

\theoremstyle{conjecture}
\newtheorem{conjecture}[theorem]{Conjecture}

\numberwithin{equation}{section}

\begin{document}

\title{Ap{\'e}ry-like sequences defined by four-term recurrence relations}

%    Information for first author
\author{Shaun Cooper}
%    Address of record for the research reported here
\address{School of Mathematical and Computational Sciences,
Massey University,
Private Bag 102904, North Shore Mail Centre,
Auckland, New Zealand }
\email{s.cooper@massey.ac.nz}

%    General info
\date{January 1, 1994 and, in revised form, June 22, 1994.}

\dedicatory{Dedicated to the memory of Richard Askey}

\keywords{
Almkvist--Zudilin numbers,
Ap{\'e}ry numbers,
asymptotic expansion,
Clausen's identity,
diagonal coefficient,
Domb numbers,
genus zero,
Lucas congruence,
modular form,
multivariate Taylor expansion,
Ramanujan's alternative theories,
self-starting sequence,
sporadic sequence,
supercongruence,
theta function identity. 
\\
\\
Status: to appear in ``Hypergeometric Functions and Their Generalizations'', Contemporary Math., AMS}
%---------------------------------------------------
\begin{abstract}
The Ap{\'e}ry numbers may be defined by a cubic three-term recurrence relation, that is, a three-term relation
where the coefficients are polynomials in the index of degree~$3$.

In this work, we first provide a systematic review of Ap{\'e}ry numbers and other related sequences that
satisfy quadratic or cubic three-term recurrence relations, and show how they are interrelated and how they
may be classified. This leads to sequences defined by cubic $k$-term recurrence relations. The cases
corresponding to $k=2$ in this framework lead
to Ramanujan's theories of elliptic functions to alternative bases, while the cases corresponding
to $k=3$ correspond to the Ap{\'e}ry, Domb, Almkvist--Zudilin numbers and other sequences that are well-studied.

We conduct a detailed analysis for the case $k=4$. Some of the sequences that arise are new.
Of particular interest are
ten sequences that are said to be self-starting in the sense that a single initial condition is
enough to start the recurrence relation. Of additional interest are two sequences which take values in~$\mathbb{Z}[i]$
and two others with values in $\mathbb{Z}[\sqrt{2}]$. Congruence properties and asymptotic expansions
for the ten self-starting sequences are
investigated and several conjectures are presented. For example, we conjecture that the
integer-valued sequence defined by
the recurrence relation
\begin{align*}
(n+1)^3T(n+1) &=2(2n+1)(5n^2+5n+2)T(n)  \\
&\qquad -8n(7n^2+1)T(n-1)+22n(2n-1)(n-1)T(n-2)
\end{align*}
and initial condition $T(0)=1$ satisfies a Lucas congruence
for every prime~$p$. Moreover, the sequence is conjectured to satisfy the supercongruence
$$
T(pn) \equiv T(n) \pmod{p^2} \quad\text{for all positive integers $n$}
$$
if $p=2,\;59$ or $5581$, and for no other primes $p<10^4$.

\end{abstract}

\maketitle

%---------------------------------------------------

\section{Introduction}
The Ap{\'e}ry numbers may be defined by the recurrence relation
\begin{equation}
\label{a2}
(n+1)^3A(n+1)=(2n+1)(17n^2+17n+5)A(n)-n^3A(n-1)
\end{equation}
and initial condition $A(0)=1$; the single initial condition is enough to start the recurrence, as can be seen by taking $n=0$ in~\eqref{a2}.
The Ap{\'e}ry numbers are famous for having been
introduced and used by R.~Ap{\'e}ry to prove that~$\zeta(3)$ is irrational, e.g., see~\cite{ape} or~\cite{vdp}.
An alternative formula for the Ap{\'e}ry numbers is the binomial sum
\begin{equation}
\label{a1}
A(n) = \sum_{k=0}^n {n \choose k}^2 {n+k \choose k}^2.
\end{equation}
This reveals that $A(n)$ is always an integer, a fact not obvious from the recurrence relation~\eqref{a2}.
The method of creative telescoping can be used to show that the $A(n)$ defined by the binomial sum~\eqref{a1}
satisfy the recurrence relation~\eqref{a2}. 
For example, using the computer algebra system Maple, and typing the commands:
\\

\begin{verbatim}
> with(sumtools):
> sumrecursion(binomial(n,k)^2*binomial(n+k,n)^2, k, A(n));
\end{verbatim}
produces the recurrence relation
$$
 \left( n-1 \right) ^{3}A \left( n-2 \right) - \left( 2\,n-1\right)\left( 17\,{n}^{2}-17
\,n+5 \right) A \left( n-1 \right) +A\left( n \right) {n}^{3}
$$
that is equal to zero. Hence, we deduce \eqref{a2} from~\eqref{a1}. 
Conversely, there does not seem to be an algorithm for determining~\eqref{a1} from~\eqref{a2} apart from for searching a database.

It is a remarkable result of F. Beukers~\cite{beukers} (see also ~\cite[Theorem 6.38]{cooperbook})
that the Ap{\'e}ry numbers are parameterised by modular forms according to the formula
\begin{equation}
\label{a3}
\prod_{j=1}^\infty \frac{(1-q^{2j})^7(1-q^{3j})^7}{(1-q^{j})^5(1-q^{6j})^5}
=\sum_{n=0}^\infty A(n)\left(q\,\prod_{j=1}^\infty \frac{(1-q^{j})^{12}(1-q^{6j})^{12}}{(1-q^{2j})^{12}(1-q^{3j})^{12}}\right)^{n},
\end{equation}
valid in a neighbourhood of $q=0$.

We briefly mention four other properties of the Ap{\'e}ry numbers. 
First, it has been shown by I.~Gessel~\cite{gessel} that
the Ap{\'e}ry numbers $A(n)$ satisfy the Lucas congruences for every prime $p$.
This means that for any positive integer $n$ and for $T(n) = A(n)$, the congruence
\begin{equation}
\label{a4}
T(n) \equiv T(n_0)T(n_1)\cdots T(n_r) \pmod p
\end{equation}
holds, where $n=n_0+n_1p+\cdots +n_rp^r$ is the $p$-adic expansion of $n$ in base $p$.

Second, the Ap{\'e}ry numbers have an asymptotic formula,
stated without proof by A.~van der Poorten~\cite{vdp}
and attributed to H.~Cohen, given by
\begin{equation}
\label{cohen}
A(n) \sim \frac{1}{2^{9/4}(\pi n)^{3/2}}\, (1+\sqrt{2})^{4n+2} \quad \mbox{as } n\rightarrow\infty.
\end{equation}
For a proof, see the work of M. Hirschhorn~\cite{mdh1}.

Third, it has been noted by various authors, e.g.,~\cite{bostan}, \cite{christol}, \cite{lipshitz}, \cite{straub} that the Ap{\'e}ry numbers occur in several ways as diagonal coefficients of multivariate Taylor expansions.
For example, A.~Straub~\cite{straub} showed that if
\begin{align*}
\lefteqn{\frac{1}{(1-x_1-x_2)(1-x_3-x_4)-x_1x_2x_3x_4}} \\
& = \sum_{n_1=0}^\infty \sum_{n_2=0}^\infty \sum_{n_3=0}^\infty \sum_{n_4=0}^\infty 
c(n_1,n_2,n_3,n_4) x_1^{n_1} x_2^{n_2} x_3^{n_3} x_4^{n_4}
\end{align*}
then
$$
c(n_1,n_2,n_3,n_4) = \sum_{k=0}^{\min\left\{n_1,n_3\right\}} {n_1 \choose k} {n_3 \choose k} {n_1+n_2-k \choose n_1} {n_3+n_4-k \choose n_3}.
$$
The diagonal coefficients are the coefficients $c(n,n,n,n)$ and are given by
$$
c(n,n,n,n) = \sum_{k=0}^n {n \choose k}^2 {2n-k \choose n}^2.
$$
Replacing $k$ with $n-k$ gives
$$
c(n,n,n,n) =  \sum_{k=0}^n {n \choose n-k}^2 {n+k \choose n}^2 = \sum_{k=0}^n {n \choose k}^2 {n+k \choose n}^2.
$$
Hence by~\eqref{a1}, the diagonal terms are the Ap{\'e}ry numbers. Further results have been
given by O.~Gorodetsky~\cite{gorodetsky}.

Finally, it was indicated by T. Sato~\cite{sato} that the Ap{\'e}ry numbers can be used to produce series for $1/\pi$, for example
$$
\frac{1}{\pi} = \left(72\sqrt{15}-160\sqrt{3}\right)\sum_{n=0}^\infty A(n) \left(\frac12-\frac{3\sqrt{5}}{20}+n\right)
\left(\frac{1-\sqrt{5}}{2}\right)^{12n}.
$$
Such formulas have been investigated in detail by H. H. Chan and H. Verrill~\cite{verrill} and further in~\cite{cooperbook}.

In view of the rich and varied properties of the Ap{\'e}ry numbers, there is interest in determining and classifying other sequences with
similar properties. 
The goal of this work is to exhibit some new sequences that satisfy recurrence relations similar to~\eqref{a2}.
One significant difference is that some of the new sequences are not integer-valued instead but take values in 
$\mathbb{Z}[i]$ or $\mathbb{Z}[\sqrt{2}]$. The modular parameterisations are established, and
Lucas type properties are conjectured. That leads to another significant difference in that some of the conjectured
congruences are for certain primes.

This work is organised as follows. Sections 2, 3 and 4 are largely expository and aim to provide
a summary of the known Ap{\'e}ry-like sequences along with their classification
and interrelations.
Thus, in Section~2 we review Zagier's six sporadic sequences and their modular parameterisations.
These sequences are integer-valued and are
defined by a three-term recurrence relation with quadratic coefficients.
Because the parameterising modular forms have weight one, we will also refer to Zagier's sequences
as the weight one examples.

In Section 3 we review the Ap{\'e}ry numbers along with other integer-valued sequences defined by
three-term recurrence relations with cubic coefficients. This leads to a theorem of Almkvist, van Straten
and Zudilin that places these sequences in one-to-one correspondence with Zagier's sequences by
means of a Clausen type identity involving the square of the generating function for the
weight one sequences.
Because the parameterising modular forms have weight two, we will refer to these sequences as weight two examples.
Key in understanding the connection between the weight one and weight two examples is the non-linear
differential equation
\begin{equation}
\label{key}
\uD^2Z - \frac{1}{2Z}\left(\uD Z\right)^2=HZ,
\end{equation}
where $\uD$ is the differential operator 
$$
\uD = BX \frac{\ud}{\ud X}
$$
and where $B^2$ and $H$ are polynomials in~$X$. In Theorem~\ref{interesting} it will be
shown that~\eqref{key} can be used to write down a simple $(k+1)$-term recurrence relation
for the coefficients $T(n)$ in the expansion
\begin{equation}
\label{keyT}
Z=\sum_{n=0}^\infty T(n)X^n,
\end{equation}
where $k$ is the maximum of the degrees of $B^2$ and $H$. It will be shown further in Section~10
how parameters in the asymptotic expansion for $T(n)$ can be obtained from
the polynomials $B^2$ and $H$.

In Section 4 we analyse a modular parameterisation of Ramanujan in detail,
and use \eqref{key} along with standard properties of theta functions to give a complete and direct proof
without using the usual inversion formula for elliptic functions or Clausen's formula.
Ramanujan's example is important because it provides a prototype for further modular parameterisations,
and this leads to further Ap{\'e}ry-type sequences which are defined by recurrence relations involving
two or more terms. A summary involving a large number of examples akin to Ramanujan's prototype
is provided in Tables~\ref{1to35} and~\ref{1to35b}.
We indicate how six of the examples in Table~\ref{1to35} are related to Zagier's weight one
sequences by another Clausen type theorem of Chan, Tanigawa, Yang and Zudilin.

The sequences in Table~\ref{1to35} that satisfy four-term recurrence relations
are studied in Sections 5--9.
Sections 5, 6, 7 and 8 provide analyses of the sequences for levels 11, 13, 14 and 15, respectively, and the 
remaining sequences are studied in Section~9. The sequences for levels 11, 14, 15 and 24 stand out in that
they are self-starting, by which we mean that the single initial condition $T(0)=1$ is enough to start the
recurrence relations. The self-starting sequences are conjectured to satisfy Lucas congruences. On the other hand,
the non self-starting sequences appear not to satisfy Lucas congruences.

The level 13 sequence is analysed in detail in Section~6. It is neither self-starting nor integer valued
but is shown to become integer-valued when multiplied by~$4^n$.
No conjectural congruences have been found for this sequence.

In Section~7, four sequences are obtained corresponding to level 14, two of which are new.
The new sequences take values in $\mathbb{Z}[\sqrt{2}]$ and are conjectured to satisfy Lucas congruences
for certain primes~$p$.

The level 15 sequences are studied on Section~8. Four sequences emerge, three
of which are new, and two of which take values in $\mathbb{Z}[i]$ and are conjectured to satisfy Lucas congruences
for certain primes~$p$.

The sequences for levels 20, 21 and 24 are discussed in Section~9.

In Section~10 we outline how parameters in the asymptotic
expansions of the coefficients $T(n)$ in \eqref{keyT} can be derived from the polynomials $B^2$ and $H$
in~\eqref{key}.

Finally, in Section~11 we consider supercongruences satisfied by the self-starting sequences defined by
four-term recurrence relations. We find that such congruences exist, but are weaker, i.e., they
are modulo $p^2$ instead
of modulo $p^3$. Moreover, such congruences seem to be fairly rare and involve only certain primes.

Throughout this work we shall make use of the notation
$$
\eta_N = q^{N/24} \prod_{j=1}^\infty (1-q^{jN})
$$
where $N$ is any positive integer and $|q|<1$. We will occasionally use the compact notation for products from the theory
of basic hypergeometric series such as
$$
(a;q)_\infty = \prod_{j=0}^\infty (1-aq^j) \quad\text{and}\quad (a,b;q)_\infty =(a;q)_\infty (b;q)_\infty
$$
so that, for example,
$$
\eta_N = q^{N/24} (q^N;q^N)_\infty.
$$

%---------------------------------------------------

\section{Zagier's sporadic sequences: the six weight one examples}
\label{S:Zagier}
An old result of
Franel~\cite{1894} is that the sequence
defined by
$$
t_c(n)=\sum_{k=0}^n {n \choose k}^3
$$
satisfies the three-term recurrence relation
\begin{equation}
\label{r2}
(n+1)^{2}t_{c}(n+1)=(7n^{2}+7n+2)t_{c}(n)+8n^{2}t_{c}(n-1), \quad n\geq 0.
\end{equation}
Observe that the coefficients in \eqref{r2} are polynomials in $n$ of degree $2$, whereas the coefficients in the
recurrence relation~\eqref{a2} for the Ap{\'e}ry numbers in~\eqref{a2} have degree~3.
Another example where the coefficients are polynomials of degree~2, also due to Ap{\'e}ry~\cite{ape}, is the sequence defined by
$$
t_{5}(n) = \sum_{k=0}^{n} {n \choose k}^{2}{n+k \choose k}.
$$
It satisfies the three-term recurrence relation
\begin{equation}
\label{r3}
(n+1)^{2}t_{5}(n+1)=(11n^{2}+11n+3)t_{5}(n)+n^{2}t_{5}(n-1), \quad n\geq 0.
\end{equation}
Motivated by the common form of the relations~\eqref{r2} and~\eqref{r3}, D. Zagier~\cite{zagier} was led to consider sequences of the form
\begin{equation}
\label{zr1}
(n+1)^2 t(n+1)=(\alpha n^2 + \alpha n+\beta)t(n)+\gamma n^2 t(n-1),  \quad n\geq 0,
\end{equation}
where $\alpha$, $\beta$ and $\gamma$ are integers. He performed a
computer search for other integral sequences that satisfy the three-term recurrence relation.
The search yielded six sequences that are not either terminating (two successive terms that are zero), polynomial,
or the product of a polynomial with a
multinomial coefficient, corresponding to the parameter values
$(\alpha,\beta,\gamma) \in S$ where
\begin{equation}
\label{Sdef}
S= \left\{ (11,3,1), (-17,-6,-72), (10,3,-9), (7,2,8), (12,4,-32), (-9,-3,-27)\right\}.
\end{equation}
They are now called Zagier's sporadic sequences.
Remarkably, all six sequences can be expressed as sums of binomial coefficients,
and they can be parameterised by modular forms. 
The details are summarised in Table~\ref{table11}.
The levels of the modular forms are listed in the first column,
and the letters (A), (B) and (C) in that column have no special
meaning other than to distinguish the three different examples that occur for level 6.
The functions $z$ in Table~\ref{table11} are modular forms of weight one, and for this reason
we shall refer to the sequences as the weight one examples.

Each function $x$ in Table~\ref{table11} satisfies the differentiation formula
\begin{equation}
\label{qdxdq}
q\frac{\ud x}{\ud q} = z^2x(1-\alpha x-\gamma x^2)
\end{equation}
and each $z$ satisfies the second order linear differential equation with respect to~$x$ given by
\begin{equation}
\label{dzdx}
\frac{\ud}{\ud x} \left(x(1-\alpha x - \gamma x^2)\frac{\ud z}{\ud x}\right) = (\beta + \gamma x)z.
\end{equation}
The proofs of \eqref{qdxdq} are case-by-case for each pair of functions $x$ and $z$ in Table~\ref{table11}. For example,
see \cite[p. 315]{cooperbook} for the case of level 5; use \cite[(6.29) and Theorem~6.22]{cooperbook} for the three level 6 cases; 
use \cite[Theorem 8.23 and (8.64)]{cooperbook} to obtain the level 8 result; and use~\cite[Theorem 9.7 and (9.16)]{cooperbook}
in the case of level~9.

The differential equation~\eqref{dzdx} is equivalent to the recurrence relation~\eqref{zr1}.
For proofs that each pair $x$ and $z$ in Table~\ref{table11} satisfies the differential equation~\eqref{dzdx},
see, for example, Theorems 5.29, 6.34, 8.28, 9.10 in~\cite{cooperbook}.

\begin{table}
\caption{Zagier's sporadic sequences, their modular parameterisations 
and entries in Sloane's OEIS \cite{sloane}.}

\vspace{-0.1cm}
The notation is $\displaystyle{z=\sum_{n=0}^\infty t(n) x^n}$, where $t(0)=1$ and

$(n+1)^{2}t(n+1)= (\alpha n^{2}+\alpha n+\beta)t(n)+\gamma\,n^{2}t(n-1) \;\text{for}\; n\geq 0.$

\vspace{0.2cm}

\centering
{\renewcommand{\arraystretch}{3.5}
{\renewcommand{\arraystretch}{2.4}
{\renewcommand{\arraystretch}{2.8}
\centering
{\resizebox{12.5cm}{!}{
\begin{tabular}{|c|c|c|c|c|c|}
\hline
Level & $(\alpha,\beta,\gamma)$ & $x$ & $z$ & $t(n)$ & OEIS \\
\hline\hline
5 & $(11,3,1)$
&$q\dfrac{(q,q^4;q^5)_\infty^5}{(q^2,q^3;q^5)_\infty^5}$ 
& $\dfrac{(q;q)_\infty^2}{(q,q^4;q^5)_\infty^5}$ 
& $\displaystyle{\sum_{j} {n \choose j}^{2}{n+j \choose j}}$ & A005258
 \\
6 (A) & $\;\;(-17,-6,-72)\;\;$ 
&$\dfrac{\eta_2\eta_6^5}{\eta_1^5\eta_3}$ 
&$\dfrac{\eta_1^6\eta_6}{\eta_2^3\eta_3^2}$ 
& $\displaystyle{\sum_{j,\ell}(-8)^{n-j}{n \choose j}{j \choose \ell}^{3}}$ & A093388
\\
6 (B) & $(10,3,-9)$
&$\dfrac{\eta_1^4\eta_6^8}{\eta_2^8\eta_3^4}$
& $\dfrac{\eta_2^6\eta_3}{\eta_1^3\eta_6^2}$ 
&$\displaystyle{\sum_{j} {n \choose j}^{2}{2j \choose j}}$ & A002893
 \\
6 (C) & $(7,2,8)$
&$\dfrac{\eta_1^3\eta_6^9}{\eta_2^3\eta_3^9}$
& $\dfrac{\eta_2\eta_3^6}{\eta_1^2\eta_6^3}$ 
&$\displaystyle{\sum_{j} {n \choose j}^{3}}$ & A000172
\\
8 & $(-12,-4,-32)$
& $\dfrac{\eta_2^2\eta_8^4}{\eta_1^4\eta_4^2}$
& $\dfrac{\eta_1^4}{\eta_2^2}$ 
&$\;\;\displaystyle{(-1)^{n}\sum_j {n \choose j}{2j \choose j}{2n-2j \choose n-j}}\;\;$ & A081085
\\
9 & $(-9,-3,-27)$
& $\dfrac{\eta_9^3}{\eta_1^3}$
& $\dfrac{\eta_1^3}{\eta_3}$ 
&${\displaystyle{\sum_{j}} (-3)^{n-3j}  {n \choose j}{n-j \choose j}{n-2j \choose j}}$ & A291898
 \\ 
\hline
\end{tabular}}}}}}
\nopagebreak[4]
\label{table11}
\end{table}

%---------------------------------------------------

\section{Ap{\'e}ry numbers and the six weight two examples}
It has already been noted that in the recurrence relation~\eqref{a2} for the Ap{\'e}ry numbers
the coefficients are polynomials of degree~3, whereas the coefficients in the recurrence relation~\eqref{zr1} for Zagier's examples,
are polynomials of degree~2. Thus, the Ap{\'e}ry numbers are not an instance of one of Zagier's examples.
It is however a remarkable fact that Zagier's examples are in one-to-one correspondence with another set of sequences which include
the Ap{\'e}ry numbers as one of the examples.
We shall now describe the correspondence.

Suppose 
\begin{equation}
\label{zseries}
z=\sum_{n=0}^\infty t(n)x^n
\end{equation}
where the sequence $\left\{t(n)\right\}$ satisfies the three-term recurrence relation
\begin{equation}
\label{trec}
(n+1)^{2}t(n+1)= (\alpha n^{2}+\alpha n+\beta)t(n)+\gamma\,n^{2}t(n-1)
\end{equation}
and initial condition $t(0)=1$. Equivalently, $z$ satisfies the second order linear differential equation
$$
\frac{\ud}{\ud x} \left(x(1-\alpha x - \gamma x^2)\frac{\ud z}{\ud x}\right) = (\beta + \gamma x)z
$$
and initial condition $z(0)=1$. Now make the change of variables
\begin{equation}
\label{wxyz1}
w=\frac{x}{1-\alpha x -\gamma x^2} \quad\text{and}\quad y=(1-\alpha x - \gamma x^2)z^2.
\end{equation}
By the chain rule, the differentiation formula~\eqref{qdxdq} in the new variables becomes
\begin{equation}
\label{qdwdq}
q\frac{\ud w}{\ud q} = y\,w\,A
\end{equation}
where
\begin{equation}
\label{Adef}
A=\sqrt{(1+\alpha w)^2+4\gamma w^2}.
\end{equation}
Moreover, applying the same change of variables to the differential
equation~\eqref{dzdx} (e.g., see \cite[Theorem 5.41]{cooperbook} for details) shows that $y$ satisfies the second
order non-linear differential equation with respect to $w$ given by
\begin{equation}
\uD^2y-\frac{1}{2y}\left(\uD y\right)^2 = Hy
\label{D2y}
\end{equation}
where $\uD $ is the differential operator defined by
$$
\uD  = A\,w\,\frac{\ud}{\ud w}
$$
and $A$ is as in~\eqref{Adef}, and $H$ is given by
\begin{equation}
\label{Hdef}
H = (2\beta - \alpha) w - \frac{(\alpha^2+4\gamma)w^2}{2}.
\end{equation}
In order to solve~\eqref{D2y} it is useful to prove the following
slightly more general result.
\begin{theorem}
\label{interesting}
Suppose $Z=Z(X)$ is analytic in a neighbourhood of $X=0$ with $Z(0)=Z_0\neq 0$, and satisfies the second order non-linear
differential equation
\begin{equation}
\label{dZdXa}
\uD^2Z - \frac{1}{2Z}\left(\uD Z\right)^2=HZ,
\end{equation}
where $\uD$ is the differential operator 
$$
\uD = BX \frac{\ud}{\ud X},
$$
where
\begin{equation}
\label{generalB}
B=\sqrt{G}\quad\text{and}\quad G = 1+\sum_{j=1}^k g_jX^j,
\end{equation}
and
\begin{equation}
\label{generalH}
H=\sum_{j=1}^k h_jX^j.
\end{equation}
Then the coefficients $T(n)$ in the power series expansion
$$
Z=\sum_{n=0}^\infty T(n)X^n
$$
satisfy the $(k+1)$-term recurrence relation
\begin{equation}
\label{generalT}
(n+1)^3 T(n+1) +(n+1)\sum_{j=1}^k g_j(n+1-{\textstyle \frac{j}{2}})(n+1-j)T(n+1-j)
\end{equation}
$$
=2\sum_{j=1}^k h_j(n+1-{\textstyle \frac{j}{2}})T(n+1-j)
$$
where $g_j$ and $h_j$ are the coefficients in the polynomials~$G$ and~$H$ in~\eqref{generalB} and~\eqref{generalH}, respectively.
\end{theorem}
\begin{proof}
Apply the operator $\uD $ to both sides of~\eqref{dZdXa} to obtain
\begin{align*}
\uD^3Z &= \frac{1}{Z} (\uD Z)(\uD^2Z) - \frac{1}{2Z^2}(\uD Z)^3 + \uD(HZ) \\
&= \frac{(\uD Z)}{Z} \left\{\uD^2Z-\frac{1}{2Z}(\uD Z)^2\right\} + H(\uD Z) + (\uD H)Z.
\intertext{The differential equation~\eqref{dZdXa} can be used to simplify the term in braces, and the result is}
\uD^3Z &= \frac{(\uD Z)}{Z} HZ + H(\uD Z) + (\uD H)Z = 2H(\uD Z) + (\uD H)Z.
\end{align*}
It follows that $Z$ satisfies the third order linear differential equation with respect to $X$ given by
$$
\uD^3Z-2H(\uD Z)-(\uD H)Z=0,
$$
or equivalently, after dividing by~$B$,
\begin{equation}
\label{D3y}
X\frac{\ud}{\ud X} \left(D^2Z\right) - 2HX\frac{\ud Z}{\ud X} - X\frac{\ud H}{\ud X} Z = 0.
\end{equation}
By direct calculation, and recalling from~\eqref{generalB} that $G=B^2$, it is easy to check that
$$
X\frac{\ud}{\ud X} \bigl(D^2Z\bigr) = G\bigl(\Delta^3Z\bigr)+\frac32\bigl(\Delta G\bigr)\bigl(\Delta^2Z\bigr) + \frac12\bigl(\Delta^2G\bigr)\bigl(\Delta Z\bigr)
$$
where
$$
\Delta = X \frac{\ud}{\ud X}.
$$
Hence the differential equation~\eqref{D3y} takes the form
$$
G\bigl(\Delta^3Z\bigr)+\frac32\bigl(\Delta G\bigr)\bigl(\Delta^2Z\bigr) + \frac12\bigl(\Delta^2G\bigr)\bigl(\Delta Z\bigr)
= 2H\bigl(\Delta Z \bigr) + \bigl(\Delta H\bigr)Z.
$$
Now substitute the expansions
$$
Z=\sum_{j=0}^\infty T(j)X^{j+r}, \quad G = 1+\sum_{j=1}^k g_jX^j\quad\text{and}\quad
H=\sum_{j=1}^k h_jX^j
$$
to obtain
\begingroup\makeatletter\def\f@size{9}\check@mathfonts
\def\maketag@@@#1{\hbox{\m@th\large\normalfont#1}}%
\begin{align}
\label{frobenius}
&\Biggl(1+\sum_{j=1}^k g_jX^j\Biggr)\Biggl(\sum_{j=0}^\infty (j+r)^3 T(j) X^{j+r}\Biggr) \\
&+\frac32\Biggl(\sum_{j=1}^k jg_jX^j\Biggr)\Biggl(\sum_{j=0}^\infty (j+r)^2 T(j) X^{j+r}\Biggr) 
+\frac12\Biggl(\sum_{j=1}^k j^2g_jX^j\Biggr)\Biggl(\sum_{j=0}^\infty (j+r) T(j) X^{j+r}\Biggr)  \nonumber \\
&\quad\quad = 2\Biggl(\sum_{j=1}^k h_jX^j\Biggr)\Biggl(\sum_{j=0}^\infty (j+r) T(j) X^{j+r}\Biggr) 
+\Biggl(\sum_{j=1}^k j h_jX^j\Biggr)\Biggl(\sum_{j=0}^\infty  T(j) X^{j+r}\Biggr). \nonumber
\end{align}
\endgroup
Extracting the coefficient of $X^r$ gives the indicial equation $r^3=0$, from which it follows that there is a unique solution that is analytic at $X=0$ and satisfies the initial condition $Z(0)=Z_0$. Finally, setting $r=0$ in~\eqref{frobenius} and extracting the
coefficient of $x^{n+1}$ leads to
\begin{align*}
&(n+1)^3T(n+1) + \sum_{j=1}^k g_j (n+1-j)^3T(n+1-j) \\
&+\frac32\sum_{j=1}^k  j g_j (n+1-j)^2 T(n+1-j)
+\frac12\sum_{j=1}^k  j^2 g_j (n+1-j) T(n+1-j) \\
&= 2\sum_{j=1}^k   h_j (n+1-j) T(n+1-j)
+\sum_{j=1}^k   j h_j T(n+1-j),
\end{align*}
where $T(n+1-j)$ is defined to be zero if $j>n+1$. This simplifies to give~\eqref{generalT}.
\end{proof}
Now we can return to solving~\eqref{D2y} for the data given by~\eqref{Adef} and~\eqref{Hdef}.
Application of Theorem~\ref{interesting} with
$$
Z=y,\quad X=w,\;\; B^2 = 1+2\alpha w +(\alpha^2+ 4\gamma) w^2\;\;\text{and}\;\;
H=(2\beta-\alpha)w-\frac{(\alpha^2+4\gamma)w^2}{2}
$$
gives
\begin{equation}
\label{yseries}
y=\sum_{n=0}^\infty s(n)w^n,
\end{equation}
where the coefficients $s(n)$ satisfy the recurrence relation
$$
(n+1)^3s(n+1) + 2\alpha (n+1)\left(n+\frac12\right)n s(n) + (\alpha^2+4\gamma)(n+1)(n)(n-1)s(n-1)
$$
$$
= 2(2\beta-\alpha)\left(n+\frac12\right)s(n)-(\alpha^2+4\gamma)ns(n-1),
$$
and this simplifies to
\begin{equation}
\label{srec}
(n+1)^3s(n+1) = -(2n+1)(\alpha n^2 + \alpha n+\alpha -2\beta)s(n) - (\alpha^2+4\gamma) n^3 s(n-1).
\end{equation}
On combining~\eqref{zseries}, \eqref{trec}, \eqref{wxyz1}, \eqref{yseries} and~\eqref{srec}, we deduce the following theorem of
G.~Almkvist, D.~van~Straten and W.~Zudilin (see~\cite{almkvist} or~\cite[p. 342]{cooperbook}):

\begin{theorem}
\label{ASZ}
Let $\left\{t(n)\right\}$ be the sequence that
satisfies the initial condition
$$
t(0)=1
$$
and the quadratic three-term recurrence relation
$$
(n+1)^2t(n+1)=(\alpha n^2+\alpha n+\beta)t(n) + \gamma n^2t(n-1)
$$
for $n\geq 0$. Then in a neighborhood of $x=0$, the following identity holds\textup:
\begin{equation}
\label{clauex1}
\left(\sum_{n=0}^\infty t(n)x^{n+\frac12}\right)^2
= \sum_{n=0}^\infty s(n)\left(\frac{x}{1-\alpha x-\gamma x^2}\right)^{n+1},
\end{equation}
where $\left\{s(n)\right\}$ satisfies the cubic three-term relation
\begin{equation}
\label{rr3}
(n+1)^3s(n+1) = -(2n+1)(\alpha n^2 + \alpha n+\alpha -2\beta)s(n) - (\alpha^2+4\gamma) n^3 s(n-1)
\end{equation}
and initial condition $s(0)=1$.
\end{theorem}
In particular, taking $(\alpha,\beta,\gamma)=(-17,-6,-72)$ in the recurrence~\eqref{rr3} gives the recurrence~\eqref{a2}
for the Ap{\'e}ry numbers, and since the initial conditions are the same, we have $s(n) = A(n)$. Hence, this
instance of the theorem gives
$$
\left(\sum_{n=0}^\infty t(n)x^{n+\frac12}\right)^2
=\sum_{n=0}^\infty s(n)\left(\frac{x}{(1+8x)(1+9x)}\right)^{n+1}
$$
where $s(n)$ are the Ap{\'e}ry numbers and $t(n)$ are the numbers in the sequence labelled~6~(A) in Table~\ref{table11}.

The six sequences $\left\{s(n)\right\}$ that correspond to the parameter values $(\alpha,\beta,\gamma)\in S$,
where $S$ is given by~\eqref{Sdef},
are all integer valued and have expressions as sums of terms involving binomial coefficients
and the details are summarised in Table~\ref{table1}.
Explicit formulas for the respective modular forms $w$ and $y$ are also given in Table~\ref{table1}. Once again, proofs
are case-by-case and can be found, for example, in Chapters 5, 6, 8 and 9 of~\cite{cooperbook}.

The functions~$z$ in Table~\ref{table11} are modular forms of weight one, while the functions~$y$ in Table~\ref{table1} are
modular forms of weight two. For that reason we shall refer to the sequences~$t(n)$ and~$s(n)$ as the
weight one and weight two examples, respectively.

\begin{table}
\caption{Ap{\'e}ry-like sequences, their modular parameterisations, 
and entries in Sloane's OEIS \cite{sloane}. The Ap{\'e}ry numbers
correspond to the sequence labelled 6(A).} 

\vspace{-0.1cm}
The notation is $\displaystyle{y=\sum_{n=0}^\infty s(n)\,w^n}$, where $s(0)=1$ and

$(n+1)^3s(n+1)=-(2n+1)(\alpha n^2+\alpha n+\alpha-2\beta)s(n) - (\alpha^2+4\gamma)n^3s(n-1)$.

\vspace{0.2cm}

{\renewcommand{\arraystretch}{2.4}
{\renewcommand{\arraystretch}{2.9}
\centering
{\resizebox{12.5cm}{!}{
{\begin{tabular}{|c|c|c|c|c|c|}
\hline
Level & $(\alpha,\beta,\gamma)$ & $w$ & $y$ & $s(n)$ & OEIS \\
\hline
$5$
&$(11,3,1)$ 
& $\displaystyle{\frac{\eta_5^6}{\eta_1^6}}$ 
& $\displaystyle{\frac{\eta_1^5}{\eta_5}}$ 
& $\displaystyle{\sum_j (-1)^{j+n}{n \choose j}^3 {4n-5j \choose 3n}}$ 
& A229111 \\
\hline
$6$ (A)
&$(-17,-6,-72)$ 
& $\displaystyle{\frac{\eta_1^{12}\eta_6^{12}}{\eta_2^{12}\eta_3^{12}}}$ 
& $\displaystyle{\frac{\eta_2^7\eta_3^7}{\eta_1^5\eta_6^5}}$  
& $\displaystyle{\sum_j {n\choose j}^2{n+j\choose j}^2}$ 
& A005259 \\
\hline
$6$ (B)
&$(10,3,-9)$ 
& $\displaystyle{\frac{\eta_2^{6}\eta_6^{6}}{\eta_1^{6}\eta_3^{6}}}$ 
& $\displaystyle{\frac{\eta_1^4\eta_3^4}{\eta_2^2\eta_6^2}}$  
& $\displaystyle{(-1)^n\sum_j {n\choose j}^2{2j\choose j}{2n-2j \choose n-j}}$ 
& A002895 \\
\hline
$6$ (C)
&$(7,2,8)$ 
& $\displaystyle{\frac{\eta_3^{4}\eta_6^{4}}{\eta_1^{4}\eta_2^{4}}}$ 
& $\displaystyle{\frac{\eta_1^3\eta_2^3}{\eta_3\eta_6}}$  
& $\displaystyle{\sum_j (-3)^{n-3j}{n+j \choose j}{n \choose j}{n-j \choose j}{n-2j \choose j}}$ 
& A125143 \\
\hline
$8$
&$(12,4,-32)$ 
& $\displaystyle{\frac{\eta_1^{8}\eta_8^{8}}{\eta_2^{8}\eta_4^{8}}}$ 
& $\displaystyle{\frac{\eta_2^6\eta_4^6}{\eta_1^4\eta_8^4}}$  
& $\displaystyle{\sum_j {n \choose j}^2 {2j \choose n}^2}$ 
& A290575 \\
\hline
$9$
&$(-9,-3,-27)$ 
& $\displaystyle{\frac{\eta_1^{6}\eta_9^{6}}{\eta_3^{12}}}$ 
& $\displaystyle{\frac{\eta_3^{10}}{\eta_1^3\eta_9^3}}$  
& $\displaystyle{\sum_{j,\ell} {n \choose j}^2 {n \choose \ell}{j \choose \ell}{j+\ell \choose n}}$ 
& A290576 \\
\hline
\end{tabular}}}}}}
\label{table1}
\end{table}

%---------------------------------------------------

\section{Jacobi's parameterisation, Ramanujan's alternative theories, and other levels}
Perhaps the oldest modular parameterisation is due to Jacobi and involves the complete elliptic integral of the first kind. 
Ramanujan took the square of Jacobi's result to get~\cite[(25)]{ramanujan_pi}
\begin{equation}
\label{rj}
\left(\frac{2K}{\pi}\right)^2 = 1 + \left(\frac12\right)^3(2kk')^2+\left(\frac{1\cdot3}{2\cdot 4}\right)^3(2kk')^4+\cdots
=\sum_{n=0}^\infty {2n \choose n}^3 \left(\frac{kk'}{4}\right)^{2n}.
\end{equation}
He then utilised $q$-expansions for $K$, $k$ and $k'$ that are originally due to Jacobi.
A complete proof of the modular parameterisation of~\eqref{rj}, e.g., see~\cite[Th.~4.30]{cooperbook}, is normally
quite long and involves first establishing Jacobi's results and then invoking Clausen's identity for
the square of a ${}_2F_1$ hypergeometric function. In view of the fundamental importance of~\eqref{rj} in Ramanujan's work,
especially in deriving series for $1/\pi$, it is instructive and useful to have a direct proof. We shall give such a proof that
uses only well-known properties of theta functions that can be found in~\cite[Ch.~16 and 17]{Part3} or~\cite[Ch. 3]{cooperbook}. 
We will write $Z$ for $(2K/\pi)^2$ and $X$ for $(kk'/4)^2$. Here is a detailed formulation of Ramanujan's result~\eqref{rj}.

\begin{theorem}
\label{T4}
Let $w$, $X$ and $Z$ be defined by
$$
w=\frac{\eta_4^8}{\eta_1^8},\quad \frac{1}{X} = \frac{1}{w}+32+256w\quad\text{and}\quad
Z=\frac{\eta_1^2\eta_4^2}{X^{5/12}}.
$$
Then:
\begin{enumerate}
\item[(a)]
Another formula for $X$ is
$$
X = \frac{\eta_1^{24}\eta_4^{24}}{\eta_2^{48}}.
$$
Hence,
$$
Z=\frac{\eta_2^{20}}{\eta_1^8\eta_4^8}\quad\text{and}\quad Z^2X = \frac{\eta_1^{8}\eta_4^{8}}{\eta_2^{8}}.
$$
\item[(b)]
The differentiation formula
$$
q\frac{\ud X}{\ud q} = Z\,X\,B
$$
holds, where $B=\sqrt{1-64X}$.
\item[(c)]
The differential equation 
$$
\uD^2Z-\frac{1}{2Z}\left(\uD Z\right)^2 = HZ
$$
holds, where 
$$
\uD = B\,X\, \frac{\ud}{\ud X} \quad\text{and}\quad H=8X,
$$
and $B$ is as above.
\item[(d)]
In a neighbourhood of $q=0$, the following power series expansion holds:
$$
Z=\sum_{n=0}^\infty {2n \choose n}^3 X^n.
$$
\end{enumerate}
\end{theorem}
\begin{proof}[Proof of theorem](a)
By the definition of $X$ we have
$$
\frac{1}{X} 
=  \left(\frac{1}{w^{1/2}}+16w^{1/2}\right)^{2}
=\left(\frac{\eta_1^4}{\eta_4^4}+16\frac{\eta_4^4}{\eta_1^4}\right)^{2}
=\frac{\eta_2^{8}}{\eta_1^{8}\eta_4^{8}}\left(\frac{\eta_1^8}{\eta_2^4}+16\frac{\eta_4^8}{\eta_2^4}\right)^{2}.
$$
The terms in parentheses can be expressed in terms of theta functions, e.g., see~\cite[Th. 3.5]{cooperbook}, to yield
$$
\frac{1}{X}  = \frac{\eta_2^{8}}{\eta_1^{8}\eta_4^{8}} \left(\phi^4(-q)+16q\psi^4(q)\right)^2,
$$
where
$$
\phi(q) = \sum_{j=-\infty}^\infty q^{j^2}\quad\text{and}\quad \psi(q) = \sum_{j=0}^\infty q^{j(j+1)/2}.
$$
By an identity of Jacobi, e.g.,~\cite[Th. 3.7]{cooperbook} this simplifies to
$$ 
\frac{1}{X}  = \frac{\eta_2^{8}}{\eta_1^{8}\eta_4^{8}} \times \phi^8(q),
$$
and on applying~\cite[Th. 3.5]{cooperbook} again to convert the theta function to an eta-quotient we obtain
$$ 
\frac{1}{X}  = \frac{\eta_2^{8}}{\eta_1^{8}\eta_4^{8}} \times \left(\frac{\eta_2^5}{\eta_1^2\eta_4^2}\right)^8 = 
\frac{\eta_2^{48}}{\eta_1^{24}\eta_4^{24}}.
$$
The result for $X$ follows on taking reciprocals, while expressions for $Z$ and $Z^2X$ are obtained by
trivially by using the result for $X$ in the definition of $Z$. This completes the proof of~(a).

Now will prove (b).
The strategy will be to first compute the derivative of~$w$ and then use the chain rule to calculate the derivative of~$X$.
We have
\begin{align}
\label{Z12}
q\frac{\ud}{\ud q} \log w 
&= q\frac{\ud}{\ud q} \log \left(q\prod_{j=1}^\infty \frac{(1-q^{4j})^8}{(1-q^j)^8}\right) \\
&= 1-32\sum_{j=1}^\infty \frac{jq^{4j}}{1-q^{4j}} + 8\sum_{j=1}^\infty \frac{jq^j}{1-q^j} \nonumber \\
&= \frac{4P(q^4)-P(q)}{3}, \nonumber
\end{align}
where
$$
P(q) = 1-24\sum_{j=1}^\infty \frac{jq^j}{1-q^j}.
$$
We now apply the sum of four squares formula~\cite[Th.~3.26]{cooperbook} followed by \cite[Th.~3.5]{cooperbook} to convert
the resulting theta function to a product, to obtain
$$
q\frac{\ud}{\ud q} \log w  = \phi^4(q) =\frac{\eta_2^{20}}{\eta_1^8\eta_4^8}.
$$
By Part (a) this simplifies further to
\begin{equation}
\label{dw}
q\frac{\ud}{\ud q} \log w = Z.
\end{equation}
For future reference, we note that~\eqref{Z12} and~\eqref{dw} imply
\begin{equation}
\label{ZX3}
Z=\frac{4P(q^4)-P(q)}{3}.
\end{equation}
Next, applying the operator $q\frac{\ud}{\ud q}$ to the relation
$$
\frac{1}{X} = \frac{1}{w} + 32 + 256w
$$
we obtain
$$
-\frac{1}{X^2} \, q\frac{\ud X}{\ud q} = \left(\frac{-1}{w^2}+256\right) q\frac{\ud w}{\ud q},
$$
and so
$$
q\frac{\ud X}{\ud q} = X^2 \left(\frac{1}{w} - 256 w\right) q\frac{\ud }{\ud q} \log w.
$$
By \eqref{dw} and by further calculation, this becomes
\begin{align*}
q\frac{\ud X}{\ud q}
&= X^2 \times \sqrt{\left(\frac{1}{w} - 256 w\right)^2} \times Z \\
&= ZX^2 \times \sqrt{\left(\frac{1}{w} + 256 w\right)^2-1024}  \\
&=ZX^2 \times \sqrt{\left(\frac{1}{X} -32 \right)^2-1024}   & \text{by the definition of $X$}\\
&=ZX^2 \times \sqrt{\frac{1}{X^2} -\frac{64}{X}}   \\
&= ZXB.
\end{align*}

Now we prove (c). First, by the chain rule and part (b) of the theorem,
the differential operator $\uD$ is related to differentiation with respect to $q$ by
\begin{equation}
\label{diffop}
\uD = BX \frac{\ud}{\ud X} = \left. BX\,q \frac{\ud}{\ud q} \right/ \left(q \frac{\ud X}{\ud q}\right) = \frac{1}{Z}\, q \frac{\ud}{\ud q}.
\end{equation}
Next, by the definition of $Z$ we have
$$
Z^{12} = \frac{\eta_1^{24}\eta_4^{24}}{X^5}.
$$
Take the logarithm and apply $q\frac{\ud}{\ud q}$ to obtain
$$
\frac{12}{Z}\, q\frac{\ud Z}{\ud q} = P(q) +4P(q^4) - 5q\frac{\ud}{\ud q} \log X.
$$
Now use the result of part (b) and~\eqref{diffop} to deduce
\begin{equation}
\label{DZ}
12 DZ +5BZ = 4P(q^4)+P(q).
\end{equation}
Next, apply the differential operator $\uD$ to~\eqref{DZ}
and make use of~\eqref{diffop} to obtain
$$
12\uD^2Z + 5\uD(BZ) = \frac{1}{Z} \, q\frac{\ud}{\ud q} \left(4P(q^4)+P(q)\right).
$$
It remains to calculate the derivative on the right hand side. On using the formula~\cite[p. 29]{cooperbook}
$$
q\frac{\ud}{\ud q}P(q) = \frac{P^2(q)-Q(q)}{12}\quad\text{where}\quad Q(q) = 1+240\sum_{j=1}^\infty \frac{j^3q^j}{1-q^j}
$$
we obtain
\begin{equation}
\label{tt}
12\uD^2Z + 5\uD(BZ) = \frac{1}{12Z} \left(16P^2(q^4)-16Q(q^4)+P^2(q)-Q(q)\right).
\end{equation}
To handle the $16Q(q^4)$ and $Q(q)$ terms, use the sum of eight squares and sum of eight
triangular numbers formulas, e.g., see \cite[Th.~3.39]{cooperbook}, to deduce
\begin{align*}
16Q(q^4)+Q(q) &= 16\phi^8(q)+\phi^8(-q)+256q^2\psi^8(q^2),
\intertext{which we rewrite as}
16Q(q^4)+Q(q) &= 16\phi^8(q)+\left(\phi^4(-q)+16q\psi^4(q^2)\right)^2-32q\phi^4(-q)\psi^4(q^2).
\intertext{By Jacobi's identity~\cite[Th. 3.7]{cooperbook} this simplifies to}
16Q(q^4)+Q(q)&= 17\phi^8(q) -32q\phi^4(-q)\psi^4(q^2).
\intertext{Now use \cite[Th. 3.5]{cooperbook} to convert the theta functions to products, then use part~(a) of the theorem to deduce}
16Q(q^4)+Q(q)&=17 \frac{\eta_2^{40}}{\eta_1^{16}\eta_4^{16}} - 32\frac{\eta_1^8\eta_4^8}{\eta_2^8} 
= 17Z^2 - 32Z^2X.
\end{align*}
Hence,~\eqref{tt} simplifies to
\begin{equation}
\label{ttt}
12\uD^2Z + 5\uD(BZ) = \frac{1}{12Z} \left(16P^2(q^4)+P^2(q)+Z^2(32X-17)\right).
\end{equation}
In order to handle the $16P^2(q^4)$ and $P^2(q)$ terms, observe that
the formulas~\eqref{ZX3} and~\eqref{DZ} can be combined to give
$$
4P(q^4) = \frac12\left(12\uD Z +5BZ+3Z\right) \quad\text{and}\quad
P(q) = \frac12\left(12\uD Z +5BZ-3Z\right).
$$
Squaring and adding gives
\begin{align*}
16P^2(q^4)+P^2(q) &= \frac12(12\uD Z +5BZ)^2 + \frac92Z^2 \\
&= 72(\uD Z)^2+60BZ(\uD Z)+\frac{25}{2}B^2Z^2+\frac92Z^2,
\end{align*}
and substituting back in~\eqref{ttt} gives
\begin{equation}
\label{tttt}
12\uD^2Z + 4\uD(BZ) = \frac{1}{12Z} \left(72(\uD Z)^2+48BZ(\uD Z)+8B^2Z^2-8Z^2\right).
\end{equation}
Now divide by 12 and rearrange terms to obtain
\begin{align*}
\uD^2Z &= \frac{1}{2Z}(\uD Z)^2 +\frac{5}{12}B(\uD Z) - \frac{5}{12} \uD(BZ)+ \frac{Z}{288}(25B^2-25+64X) \\
&= \frac{1}{2Z}(\uD Z)^2 + 8XZ
\end{align*}
where the last step follows by simple calculation using the definitions of $B$ and~$\uD$.
This completes the proof of~(c).

Finally we prove (d). By the result of (c) and application of Theorem~\ref{interesting} with $B^2=1-64X$ and $H=8X$
it follows that the coefficients $T(n)$ in the expansion $Z=\sum_{n=0}^\infty T(n)X^n$
satisfy the two-term recurrence relation
$$
(n+1)^3T(n+1)+(n+1)(-64)\left(n+\mbox{$\frac12$}\right)nT(n) = 16\left(n+\mbox{$\frac12$}\right)T(n)
$$
which simplifies to
$$
(n+1)^3T(n+1)=8(2n+1)^3T(n).
$$
On solving the recurrence and noting that $T(0)=1$ because $Z=1$ and $X=0$ when $q=0$, we deduce that
$$
T(n) = {2n \choose n}^3.
$$
This completes the proof of Theorem~\ref{T4} and establishes the modular parameterisation
in the formula~\eqref{rj} that was used by Ramanujan.
\end{proof}

To give another example in explicit detail, here is the result for level~$10$.
\begin{theorem}
\label{T10}
Let $w$, $X$ and $Z$ be defined by
$$
w=\frac{\eta_2^{4}\eta_{10}^{4}}{\eta_1^{4}\eta_5^{4}}, \quad
\frac{1}{X} = \frac{1}{w}+8+16w\quad\text{and}\quad Z=\frac{\eta_1\eta_2\eta_5\eta_{10}}{X^{3/4}}.
$$
Then:
\begin{enumerate}
\item[(a)]
The differentiation formula
$$
q\frac{\ud X}{\ud q} = Z\,X\,B
$$
holds, where $B = \sqrt{(1+4X)(1-16X)}$.
\item[(b)]
The differential equation 
$$
\uD^2Z-\frac{1}{2Z}\left(\uD Z\right)^2 = HZ
$$
holds, where 
$$
\uD = B\,X\, \frac{\ud}{\ud X} \quad\text{and}\quad H=2X(1+15X),
$$
and $B$ is as above.
\item[(c)]
The coefficients $T(n)$ in the expansion $Z=\sum_{n=0}^\infty T(n)X^n$ satisfy the three-term recurrence
relation
$$
(n+1)^3T(n+1) = 2(2n+1)(3n^2+3n+1)T(n)+4n(16n^2-1)T(n-1)
$$
and initial condition $T(0)=1$, the single initial condition being sufficient to start the recurrence.
Hence, in a neighbourhood of $q=0$, the following power series expansion holds:
$$
Z=\sum_{n=0}^\infty \left\{\sum_{j=0}^n {n \choose j}^4\right\} X^n.
$$
\end{enumerate}
\end{theorem}
\begin{proof}
Parts (a) and (c) follow from Theorems~10.15, 10.17, 10.19 and~10.24 in~\cite{cooperbook}.
Part~(b) follows from Theorem~10.20 in~\cite{cooperbook} by change of variable.
A different proof of~(c) has been given in~\cite{ctyz}. 
\end{proof} 
It may be remarked that the three-term recurrence relation in Part~(c) of Theorem~\ref{T10}
follows immediately from the result of Part~(b) by application of Theorem~\ref{interesting}.
Moreover, following the comments after~\eqref{a1},
a computer algebra system can be used to show that the sequence defined by
$$
T(n) = \sum_{j=0}^n {n \choose j}^4
$$
is a solution to the recurrence relation.

The examples in Theorems~\ref{T4} and \ref{T10} have some key aspects in common:
\begin{itemize}
\item[(i)]
They start with a modular function $X$ and a weight two modular form $Z$. The weight two modular form
is normally defined by
$$
Z=\left(\prod_{d|\ell} \eta_d^{e_d}\right)\times \frac{1}{X^{m}}
$$
where $d$ ranges over the divisors of the level~$\ell$, the exponents $e_d$ are integers with
$\sum_{d|\ell} e_d = 4$, and
$$
m = \frac{1}{24} \sum_{d|\ell} de_d.
$$
\item[(ii)] There is a polynomial $G=G(X)$ with $G(0)=1$
for which the differentiation formula
\begin{equation}
\label{qdXdq1}
q\frac{\ud X}{\ud q} =BXZ
\end{equation}
holds with $B=\sqrt{G}$.
\item[(iii)] There is a polynomial $H=H(X)$ with $H(0)=0$ for which the differential equation
\begin{equation}
\label{D2Z}
\uD^2Z - \frac{1}{2Z}\left(\uD Z\right)^2=HZ
\end{equation}
holds, where $\uD$ is the differential operator 
$$
\uD = BX \frac{\ud}{\ud X}
$$
and $B$ is as for (ii).
\item[(iv)] If $k$ is the maximum of the degrees of the polynomials $G$ and $H$, then the coefficients~$T(n)$ in the
expansion $Z=\sum_{n=0}^\infty T(n)X^n$ satisfy a $(k+1)$-term recurrence relation given by~\eqref{generalT}. 
\end{itemize}
Further examples of functions $X$ and $Z$ satisfying (i)--(iv) are given in~Tables~\ref{1to35} and~\ref{1to35b} below.
The results of Theorems~\ref{T4} and~\ref{T10} appear in the tables as the entries corresponding to levels~4 and~10,
respectively. Tables~\ref{1to35} and~\ref{1to35b} extend~\cite[Tables~14.1 and~14.2]{cooperbook}. 
Explicit formulas for $T(n)$ in terms of sums of binomial coefficients, where such formulas are known, are
listed in Table~\ref{1to35}. Otherwise, by Theorem~\ref{interesting}, the number of
terms in the recurrence relation is one more than the maximum of the degrees
of~$B^2$ and $H$, and this is recorded in the last column of the table.

In principle it is possible
to extend Tables~\ref{1to35} and~\ref{1to35b} further to other subgroups of $\text{SL}_2(\mathbb{R})$ of genus zero as per the following result
which we quote exactly from Zagier's survey~\cite[Prop. 21]{zagierelliptic}. 
\begin{theorem}
\label{TZ}
Let $f(z)$ be a (holomorphic or meromorphic) modular form of positive weight~$k$
on some group $\Gamma$ and $t(z)$ a modular function with respect to~$\Gamma$.
Express $f(z)$ (locally) as $\oldphi(t(z))$. Then the function $\oldphi(t)$ satisfies a
linear differential equation of order~$k + 1$ with algebraic coefficients,
or with polynomial coefficients if~$\Gamma \setminus \mathfrak{H}$ has genus $0$ and $t(z)$ generates
the field of modular functions on~$\Gamma$.
\end{theorem}

The results in Theorems~\ref{T4} and~\ref{T10}, and indeed all of the entries in Tables~\ref{1to35} and~\ref{1to35b},
correspond to particular and explicit instances of Zagier's Theorem~\ref{TZ}
in the case~$k=2$. The entries in Tables~\ref{table11} and~\ref{table1}
are particular instances of the cases $k=1$ and $k=2$, respectively, of Zagier's Theorem~\ref{TZ}.

\begin{table}
\caption{Data for $\displaystyle{Z=\sum_{n=0}^\infty T(n)X^n}. \qquad$
The notation is: \\
$\displaystyle{\scalemath{.9}{Q=1+240\sum_{j=1}^\infty \frac{j^3 q^{j}}{1-q^j},}}\quad$ 
$\displaystyle{\scalemath{.9}{R=1-504\sum_{j=1}^\infty \frac{j^5q^{j}}{1-q^j},}}$ \\
$\displaystyle{\scalemath{.9}{\eta_N=q^{N/24}\prod_{j=1}^\infty (1-q^{jN}),}}\quad$ 
$\displaystyle{\scalemath{.9}{\theta_{a,b,c}=\sum_{j=-\infty}^\infty \sum_{k=-\infty}^\infty q^{aj^2+bjk+ck^2}}}$}
{\renewcommand{\arraystretch}{2.9}
\centering
{\resizebox{12.5cm}{!}{
\begin{tabular}{|c|c|c|c|c|}
\hline
level & $w$ & $X$ & $Z$ & $T(n)$ \\
\hline \hline
$1$ 
& $\frac{1}{432}\left(\frac{Q^{3/2}-R}{Q^{3/2}+R}\right)$ 
& $\displaystyle{\frac{w}{(1+432w)^2}}$ 
& $\displaystyle{\frac{\eta_1^{4}}{X^{1/6}}}$
& $\displaystyle{{6n \choose 3n}{3n \choose n}{2n \choose n}}$ 
 \\
\hline
$2$ 
& $\displaystyle{\frac{\eta_2^{24}}{\eta_1^{24}}}$
& $\displaystyle{\frac{w}{(1+64w)^2}}$ 
& $\displaystyle{\frac{\eta_1^{2}\eta_2^2}{X^{1/4}}}$ 
& $\displaystyle{{4n \choose 2n}{2n \choose n}^2}$ 
\\
\hline
$3$ 
& $\displaystyle{\frac{\eta_3^{12}}{\eta_1^{12}}}$
& $\displaystyle{\frac{w}{(1+27w)^2}}$ 
& $\displaystyle{\frac{\eta_1^{2}\eta_3^2}{X^{1/3}}}$ 
& $\displaystyle{{3n \choose n}{2n \choose n}^2}$ 
\\
\hline
$4$ 
& $\displaystyle{\frac{\eta_4^{8}}{\eta_1^{8}}}$
& $\displaystyle{\frac{w}{(1+16w)^2}}$ 
& $\displaystyle{\frac{\eta_1^2\eta_4^2}{X^{5/12}}}$ 
& $\displaystyle{{2n \choose n}^3}$ 
\\
\hline
$5$ 
& $\displaystyle{\frac{\eta_5^{6}}{\eta_1^{6}}}$
& $\displaystyle{\frac{w}{1+22w+125w^2}}$ 
& $\displaystyle{\frac{\eta_1^{2}\eta_5^2}{X^{1/2}}}$ 
& $\displaystyle{{2n \choose n}\sum_{j} {n \choose j}^{2}{n+j \choose j}}$ 
\\
\hline
$6$ (A)
& $\displaystyle{\frac{\eta_1^{12}\eta_6^{12}}{\eta_2^{12}\eta_3^{12}}}$
& $\displaystyle{\frac{w}{1-34w+w^2}}$ 
& $\displaystyle{\frac{\eta_1\eta_2\eta_3\eta_6}{X^{1/2}}}$ 
&  $\displaystyle{{2n \choose n}\sum_{j,\ell}(-8)^{n-j}{n \choose j}{j \choose \ell}^{3}}$ 
\\
\hline
$6$ (B)
& $\displaystyle{\frac{\eta_2^{6}\eta_6^{6}}{\eta_1^{6}\eta_3^{6}}}$
& $\displaystyle{\frac{w}{1+20w+64w^2}}$ 
& $\displaystyle{\frac{\eta_1\eta_2\eta_3\eta_6}{X^{1/2}}}$ 
& $\displaystyle{{2n \choose n}\sum_{j} {n \choose j}^{2}{2j \choose j}}$
\\
\hline
$6$ (C)
& $\displaystyle{\frac{\eta_3^{4}\eta_6^{4}}{\eta_1^{4}\eta_2^{4}}}$
& $\displaystyle{\frac{w}{1+14w+81w^2}}$ 
& $\displaystyle{\frac{\eta_1\eta_2\eta_3\eta_6}{X^{1/2}}}$ 
& $\displaystyle{{2n \choose n}\sum_{j} {n \choose j}^{3}}$ 
\\
\hline
$7$ 
& $\displaystyle{\frac{\eta_7^{4}}{\eta_1^{4}}}$
& $\displaystyle{\frac{w}{1+13w+49w^2}}$ 
& $\displaystyle{\frac{\eta_1^{2}\eta_7^2}{X^{2/3}}}$ 
& $\displaystyle{\sum_{j} {n \choose j}^{2} {2j \choose n} {n+j \choose j} }$ 
\\
\hline
$8$ 
& $\displaystyle{\frac{\eta_1^{8}\eta_8^{8}}{\eta_2^{8}\eta_4^{8}}}$
& $\displaystyle{\frac{1}{1-24w+16w^2}}$ 
& $\displaystyle{\frac{\eta_2^{2}\eta_4^2}{X^{1/2}}}$ 
& $\;\;\displaystyle{{2n \choose n}(-1)^{n}\sum_j {n \choose j}{2j \choose j}{2n-2j \choose n-j}}\;\;$ 
\\
\hline
$9$ 
& $\displaystyle{\frac{\eta_1^{6}\eta_9^{6}}{\eta_3^{12}}}$
& $\displaystyle{\frac{w}{1-18w-27w^2}}$ 
& $\displaystyle{\frac{\eta_3^4}{X^{1/2}}}$ 
& $\displaystyle{{2n \choose n}{\sum_{j}} (-3)^{n-3j} {n \choose j}{n-j \choose j}{n-2j \choose j}}$ 
\\
\hline
$10$ 
& $\displaystyle{\frac{\eta_2^{4}\eta_{10}^{4}}{\eta_1^{4}\eta_5^{4}}}$
& $\displaystyle{\frac{w}{(1+4w)^2}}$ 
& $\displaystyle{\frac{\eta_1\eta_2\eta_5\eta_{10}}{X^{3/4}}}$ 
& $\displaystyle{\sum_j {n \choose j}^4}$ 
\\
\hline
$11$ 
& 
& $\displaystyle{\left(\frac{\eta_1\eta_{11}}{\theta_{1,1,3}}\right)^2}$ 
& $\displaystyle{\frac{\eta_1^2\eta_{11}^2}{X}}$ 
& 4-term recurrence relation
\\
\hline
$12$ 
& $\displaystyle{\frac{\eta_1^{4}\eta_{12}^{4}}{\eta_3^{4}\eta_4^{4}}}$
& $\displaystyle{\frac{w}{(1+w)^2}}$ 
& $\displaystyle{\frac{\eta_1\eta_3\eta_4\eta_{12}}{X^{5/6}}}$ 
& $\displaystyle{\sum_j {n \choose j}^2{2j \choose j}{2n-2j \choose n-j}}$ 
\\
\hline
\end{tabular}}}}
\vspace{2in}
\label{1to35}
\end{table}

\setcounter{table}{2}

\begin{table*}
\caption{Continued.}
{\renewcommand{\arraystretch}{2.9}
\centering
{\resizebox{12.5cm}{!}{
\begin{tabular}{|c|c|c|c|c|}
\hline
level & $w$ & $X$ & $Z$ & $T(n)$ \\
\hline \hline
$13$ 
& $\displaystyle{\frac{\eta_{13}^{2}}{\eta_1^{2}}}$
& $\displaystyle{\frac{w}{1+5w+13w^2}}$ 
& $\displaystyle{\frac{\eta_1^2\eta_{13}^2}{X^{7/6}}}$
& 4-term recurrence relation 
\\
\hline
$14$ (A)
& $\displaystyle{\frac{\eta_1^{4}\eta_{14}^{4}}{\eta_2^{4}\eta_7^{4}}}$
& $\displaystyle{\frac{w}{(1-w)^2}}$ 
& $\displaystyle{\frac{\eta_1\eta_2\eta_7\eta_{14}}{X}}$ 
& $\displaystyle{\sum_{j,k} {n+j \choose 2j+2k}{2j+2k \choose j+k}{2k \choose k}^2 {k \choose j}}$ 
\\
\hline
$14$ (B)
& $\displaystyle{\frac{\eta_1^{4}\eta_{14}^{4}}{\eta_2^{4}\eta_7^{4}}}$
& $\displaystyle{\frac{w}{(1+w)^2}}$ 
& $\displaystyle{\frac{\eta_1\eta_2\eta_7\eta_{14}}{X}}$ 
& 4-term recurrence relation
\\
\hline
$15$ (A)
& $\displaystyle{\frac{\eta_3^{2}\eta_{15}^{2}}{\eta_1^{2}\eta_5^{2}}}$
& $\displaystyle{\frac{w}{(1+3w)^2}}$ 
& $\displaystyle{\frac{\eta_1\eta_3\eta_5\eta_{15}}{X}}$ 
& 4-term recurrence relation
\\
\hline
$15$ (B)
& $\displaystyle{\frac{\eta_3^{2}\eta_{15}^{2}}{\eta_1^{2}\eta_5^{2}}}$
& $\displaystyle{\frac{w}{(1-3w)^2}}$ 
& $\displaystyle{\frac{\eta_1\eta_3\eta_5\eta_{15}}{X}}$ 
& 4-term recurrence relation
\\
\hline
$18$ 
& $\displaystyle{\frac{\eta_1^{2}\eta_2^{2}\eta_9^2\eta_{18}^2}{\eta_3^{4}\eta_6^{4}}}$
& $\displaystyle{\frac{w}{(1+3w)^2}}$ 
& $\displaystyle{\frac{\eta_3^{2}\eta_6^2}{X^{3/4}}}$ 
& $\displaystyle{\sum_j (-1)^j {n \choose j} {2j \choose j} {2n-2j \choose n-j} {2n-3j \choose n}}$ 
\\
\hline
$20$ 
& $\displaystyle{\frac{\eta_1^{2}\eta_{20}^{2}}{\eta_4^{2}\eta_5^{2}}}$
& $\displaystyle{\frac{w}{(1+w)^2}}$ 
& $\displaystyle{\frac{\eta_2^2\eta_{10}^2}{X}}$ 
& 4-term recurrence relation
\\
\hline
$21$ 
& $\displaystyle{\frac{\eta_1^{2}\eta_{21}^{2}}{\eta_3^{2}\eta_7^{2}}}$
& $\displaystyle{\frac{w}{(1-w)^2}}$ 
& $\displaystyle{\frac{\eta_1\eta_3\eta_7\eta_{21}}{X^{4/3}}}$ 
& 4-term recurrence relation
\\
\hline
$22$ 
& $\displaystyle{\frac{\eta_2^{2}\eta_{22}^{2}}{\eta_1^{2}\eta_{11}^{2}}}$
& $\displaystyle{\frac{w}{(1+2w)^2}}$ 
& $\displaystyle{\frac{\eta_1\eta_2\eta_{11}\eta_{22}}{X^{3/2}}}$ 
& 5-term recurrence relation
\\
\hline
$23$
& 
& $\displaystyle{\frac{2\eta_1\eta_{23}}{\theta_{1,1,6}+\theta_{2,1,3}}}$ 
& $\displaystyle{\frac{\eta_1^2\eta_{23}^2}{X^2}}$ 
& 7-term recurrence relation
\\
\hline
$24$ 
& $\displaystyle{\frac{\eta_1^{2}\eta_{3}^{2}\eta_8^{2}\eta_{24}^{2}}{\eta_2^{2}\eta_{4}^{2}\eta_6^{2}\eta_{12}^{2}}}$
& $\displaystyle{\frac{w}{1+4w^2}}$ 
& $\displaystyle{\frac{\eta_2\eta_4\eta_{6}\eta_{12}}{X}}$ 
& $\displaystyle{\sum_j {n \choose 2j}{2j \choose j}^2{2n-4j \choose n-2j}}$ 
\\
\hline
$33$ 
& $\displaystyle{\frac{\eta_3\eta_{33}}{\eta_1\eta_{11}}}$
& $\displaystyle{\frac{w}{1+w+3w^2}}$ 
& $\displaystyle{\frac{\eta_1\eta_3\eta_{11}\eta_{33}}{X^{2}}}$ 
& 6-term recurrence relation
\\
\hline
$35$ 
& $\displaystyle{\frac{\eta_1\eta_{35}}{\eta_5\eta_7}}$
& $\displaystyle{\frac{w}{1+w-w^2}}$ 
& $\displaystyle{\frac{\eta_1\eta_5\eta_7\eta_{35}}{X^{2}}}$ 
& 6-term recurrence relation
\\
\hline
\end{tabular}}}}
\vspace{2in}
\end{table*}

\begin{table}
\caption{Data for $B^2$ and $H$, where $X$ and $Z$ are as for Table~\ref{1to35}, $\displaystyle{B=\frac{q}{ZX} \frac{\ud X}{\ud q}},\quad$
$\displaystyle{D^2Z-\frac{1}{2Z}(DZ)^2=HZ,\quad}$ and $\;\displaystyle{D=BX \frac{\ud}{\ud X}}$.}
{\renewcommand{\arraystretch}{2.9}
\centering
{\resizebox{12.5cm}{!}{
\begin{tabular}{|c|c|c|l|}
\hline
level & $B^2$ & $H$ & References \\
\hline \hline
$1$ 
& $\displaystyle{1-1728X}$
& $\displaystyle{120X}$ 
& \cite{cooperbook}, \cite{ramanujan_pi}
 \\
\hline
$2$ 
& $\displaystyle{1-256X}$
& $\displaystyle{24X}$ 
& \cite{cooperbook}, \cite{ramanujan_pi}
\\
\hline
$3$ 
& $\displaystyle{1-108X}$
& $\displaystyle{12X}$ 
& \cite{cooperbook}, \cite{ramanujan_pi}
\\
\hline
$4$ 
& $\displaystyle{1-64X}$
& $\displaystyle{8X}$ 
& \cite{cooperbook}, \cite{ramanujan_pi}
\\
\hline
$5$ 
& $\displaystyle{1-44X-16X^2}$
& $\displaystyle{6X(1+X)}$ 
& \cite{ctyz}, \cite{cooperbook}
\\
\hline
$6$ (A)
& $\displaystyle{(1+32X)(1+36X)}$
& $\displaystyle{-12X(1+36X)}$ 
& \cite{ctyz}, \cite{cooperbook}
\\
\hline
$6$ (B)
& $\displaystyle{(1-4X)(1-36X)}$
& $\displaystyle{6X(1-9X)}$ 
& \cite{ctyz}, \cite{cooperbook}
\\
\hline
$6$ (C)
& $\displaystyle{(1+4X)(1-32X)}$
& $\displaystyle{4X(1+12X)}$ 
& \cite{ctyz}, \cite{cooperbook}
\\
\hline
$7$ 
& $\displaystyle{(1+X)(1-27X)}$
& $\displaystyle{4X(1+3X)}$ 
& \cite{sporadic}, \cite{cooperbook}
\\
\hline
$8$ 
& $\displaystyle{(1+16X)(1+32X)}$
& $\displaystyle{-8X(1+24X)} $
& \cite{ctyz}, \cite{cooperbook}
\\
\hline
$9$ 
& $\displaystyle{1+36X+432X^2}$
& $\displaystyle{-6X(1+27X)}$ 
& \cite{ctyz}, \cite{cooperbook}
\\
\hline
$10$ 
& $\displaystyle{(1+4X)(1-16X)}$
& $\displaystyle{2X(1+15X)}$ 
& \cite{ctyz}, \cite{cooper10a}, \cite{cooperbook}
\\
\hline
$11$ 
& $\quad\qquad\displaystyle{1-20X+56X^2-44X^3}\quad\qquad$
& $\qquad\qquad\displaystyle{4X(1-8X+11X^2)}\qquad\qquad$ 
& \cite{cooperbook}, \cite{ge}
\\
\hline
$12$ 
& $\displaystyle{(1-4X)(1-16X)}$
& $\displaystyle{4X(1-8X)}$ 
& \cite{cooperbook}, \cite{cooperye12}
\\
\hline
\end{tabular}}}}
\vspace{2in}
\label{1to35b}
\end{table}

\setcounter{table}{3}

\begin{table}
\caption{Continued}
{\renewcommand{\arraystretch}{2.9}
\centering
{\resizebox{12.5cm}{!}{
\begin{tabular}{|c|c|c|l|}
\hline
level & $B^2$ & $H$ & References \\
\hline \hline
$13$ 
& $\displaystyle{(1+X)(1-10X-27X^2)}$ 
& $\displaystyle{\frac{X}{8}(12+175X+231X^2)}$
& \cite{cooperye13}
\\
\hline
$14$ (A)
& $\displaystyle{(1+4X)(1-10X-7X^2)}$ 
& $\displaystyle{\frac{X}{2}(2+51X+56X^2)}$
& \cite{alchemy}
\\
\hline
$14$ (B)
& $\displaystyle{(1-4X)(1-18X+49X^2)}$ 
& $\displaystyle{\frac{X}{2}(10-141X+392X^2)}$
& \cite{cooperye}, \cite{guilleratranslation} $\qquad$
\\
\hline
$15$ (A)
& $\displaystyle{(1-12X)(1-2X+5X^2)}$ 
& $\displaystyle{\frac{3X}{2}(2-11X+40X^2)}$
& \cite{cooperye}
\\
\hline
$15$ (B)
& $\displaystyle{(1+12X)(1+22X+125X^2)}$ 
& $\displaystyle{-\frac{3X}{2}(2+25X)(3+40X)}$
& \cite{cooperNTDU}
\\
\hline
$18$ 
& $\displaystyle{(1-12X)(1-16X)}$ 
& $\displaystyle{6X(1-15X)}$
& \cite{sporadic}, \cite{cooperbook}
\\
\hline
$20$ 
& $\displaystyle{(1-4X)(1-12X+16X^2)}$ 
& $\displaystyle{4X(1-10X+18X^2)}$
& \cite{huber20}
\\
\hline
$21$ 
& $\displaystyle{(1+4X)(1-2X-27X^2)}$ 
& $\displaystyle{-\frac{X}{2}(2-47X-240X^2)}$
& \cite{anusha}
\\
\hline
$22$ 
& $\displaystyle{(1-8X)(1-4X^2+4X^3)}$ 
& $\displaystyle{2X(1-3X)(1+4X-10X^2)}$
& \cite{anusha}
\\
\hline
$23$
& $\displaystyle{(1-X^2+X^3)}(1-8X+3X^2-7X^3)$ 
& $\displaystyle{2X(1-X-X^2+12X^3-15X^4+14X^5)}$
& \cite{ge}
\\
\hline
$24$ 
& $\displaystyle{(1+4X)(1-4X)(1-8X)}$ 
& $\displaystyle{2X(1+5X-64X^2)}$
& \cite{alchemy}, \cite{yeEmail}
\\
\hline
$33$ 
& $\displaystyle{(1-2X-11X^2)(1+4X+8X^2+4X^3)}$ 
& $\displaystyle{-\frac{X}{2}(2-15X-152X^2-404X^3-264X^4)}$
& \cite{anusha}
\\
\hline
$35$ 
& $\displaystyle{(1-2X+5X^2)(1-8X+16X^2-28X^3)}$ 
& $\displaystyle{\frac{X}{2}(6-61X+296X^2-580X^3+840X^4)}$
& \cite{anusha}
\\
\hline
\end{tabular}}}}
\vspace{3in}
\end{table}
The entries for $w$ in Tables~\ref{table1} and \ref{1to35} are identical for the levels labelled 5, 6(A), 6(B), 6(C), 8 and 9.
Closer inspection of Tables~\ref{table1} and \ref{1to35} shows that for those levels, $X$ is related to $w$ by the change of variable
\begin{equation}
\label{Xcov}
X=\frac{w}{(1+\alpha w)^2 + 4\gamma w^2},
\end{equation}
where $(\alpha,\beta,\gamma)$ is the appropriate triple from the set~$S$ in~\eqref{Sdef}. Moreover, the corresponding
modular forms~$Z$ in Table~\ref{1to35} is related to the corresponding modular form~$y$ in Table~\ref{table1} by
\begin{equation}
\label{Zcov}
X^{1/2}Z = w^{1/2}y.
\end{equation}
Applying this change of variables to the formulas~\eqref{qdwdq} and~\eqref{Adef}
leads to the differentiation formula
$$
q\frac{\ud X}{\ud q} = ZXB
$$
and the differential equation
\begin{equation}
\label{Ztilde}
\tilde{\uD}^2Z-\frac{1}{2Z}\left(\tilde{\uD} Z\right)^2 = \tilde{H}Z,
\end{equation}
where
\begin{equation}
\label{Btilde}
B^2 = 1-4\alpha X - 16\gamma X^2,\quad \tilde{H} = 2X(\beta+3\gamma X)\quad\text{and}\quad \tilde{\uD} = BX\frac{\ud}{\ud X}.
\end{equation}
Application of Theorem~\ref{interesting}, with $B^2$ and $\tilde{H}$ as above, shows that the coefficients $T(n)$ in the
expansion $Z=\sum_{n=0}^\infty T(n)X^n$ satisfy the three term recurrence relation
\begin{equation}
\label{cubic3term}
(n+1)^3 T(n+1) = 2(2n+1)(\alpha n^2+\alpha n+\beta)T(n) + 4\gamma n(4n^2-1)T(n-1).
\end{equation}
But it is easy to check that if the coefficients $t(n)$ satisfy the quadratic three-term recurrence relation~\eqref{zr1},
then the coefficients ${2n \choose n} t(n)$ will satisfy the cubic three-term recurrence relation~\eqref{cubic3term}. 
It follows that for each of the six parameter sets $(\alpha,\beta,\gamma)\in S$ we have
$$
T(n) = {2n \choose n} t(n)
$$
as can be seen in Table~\ref{1to35} for the levels labelled 5, 6(A), 6(B), 6(C), 8 and 9.
The discussion above provides proof of the following result of
of H.~H.~Chan, Y.~Tanigawa, Y.~Yang and W.~Zudilin (see~\cite{ctyz} or \cite[p. 342]{cooperbook}):
\pagebreak[4]
\begin{theorem}
Let $\left\{t(n)\right\}$ be the sequence that
satisfies the initial condition
$t(0)=1$
and the three-term recurrence relation
$$
(n+1)^2t(n+1)=(\alpha n^2+\alpha n+\beta)t(n) + \gamma n^2t(n-1)
$$
for $n\geq 0$. Then in a neighborhood of $x=0$, the following identity holds\textup:
\begin{equation}
\label{clauex2}
\left(\sum_{n=0}^\infty t(n)x^{n}\right)^2
= \frac{1}{1+\gamma x^2} \sum_{n=0}^\infty {2n \choose n}t(n)
\left(\frac{x(1-\alpha x - \gamma x^2)}{(1+\gamma x^2)^2}\right)^{n}.
\end{equation}
\end{theorem}
\begin{proof}
Note that the change of variables \eqref{wxyz1} along with~\eqref{Xcov} and~\eqref{Zcov}
can be combined to give the overall change of variable
$$
Z=(1+\gamma x^2)z^2\quad\text{and}\quad
X=\frac{x(1-\alpha x - \gamma x^2)}{(1+\gamma x^2)^2}.
$$
Then apply Theorem~\ref{interesting} for the data in~\eqref{Ztilde} and~\eqref{Btilde}.
\end{proof}

As already mentioned, the data in Tables~\ref{1to35} and~\ref{1to35b} for level 4 corresponds to
Ramanujan's example worked out in detail in Theorem~\ref{T4}. It is called the classical theory.
The data for levels 1, 2 and 3 correspond to what are called Ramanujan's theories of elliptic functions
to alternative bases, e.g., see~\cite{bhargava} or~\cite[Chapter 4]{cooperbook}. 

The recurrence relations for the sequences $T(n)$ in Table~\ref{1to35} for levels 7, 10 and 18 are
of the form
\begin{equation}
\label{scr}
(n+1)^3T(n+1) = (2n+1)(a n^2 + a n+b)T(n) +n(cn^2+d)T(n-1)
\end{equation}
for $(a,b,c,d) = (6,2,64,-4)$, $(13,4,27,-3)$ and $(14,6,-192,12)$, respectively.
The recurrence~\eqref{scr} is a generalisation of the recurrences~\eqref{rr3} and~\eqref{cubic3term}.
Precisely, the recurrences~\eqref{rr3} and~\eqref{cubic3term} correspond to the special cases
$d=0$ and $c=-4d$ respectively of~\eqref{rr3}, and
that is why the level 7, 10 and 18 sequences do not occur in Zagier's examples in Section~\ref{S:Zagier}.
The sequences $T(n)$ nevertheless have interesting number theoretic properties. For example, A.~Malik and
A.~Straub~\cite{malikstraub} showed that the sequences $T(n)$ in Table~\ref{1to35} corresponding to levels 7, 10 and 18, as well as
all of the sequences $t(n)$ and $s(n)$ in Tables~\ref{table11} and \ref{table1}, all
satisfy Lucas congruences of the form~\eqref{a4} for every prime~$p$.

All of the sequences $\left\{T(n)\right\}$ in Table~\ref{1to35} are integer valued, except for the sequence corresponding to
level~13 which will be discussed in detail in Section~\ref{Level13}. Before doing that, the level~11 theory will be
discussed in Section~\ref{Level11}.
In Sections~\ref{14special} and~\ref{154} a more detailed investigation is performed for the level~14 and level~15 theories.
This will result in new Ap{\'e}ry-like sequences that take values in~$\mathbb{Z}[\sqrt{2}]$ for level~14,
and values in $\mathbb{Z}[i]$ for level~15. The remaining sequences from Table~\ref{1to35} that satisfy four-term
recurrence relations are investigated for Lucas congruences in Section~\ref{level202124}.

%---------------------------------------------------

\section{The level $11$ theory}
\label{Level11}
In this section we analyse the data in Tables~\ref{1to35} and~\ref{1to35b} for level 11 in detail and make a conjecture.
According to those tables, Theorem~\ref{interesting} can be applied to the functions
$$
Z=\left(\sum_{j=-\infty}^\infty \sum_{k=-\infty}^\infty q^{j^2+jk+3k^2}\right)^2\quad\text{and}\quad
X=\frac{\eta_1^2\eta_{11}^2}{Z}
$$
with data given by
$$
B^2 = 1-20X+56X^2-44X^3\quad\text{and}\quad
H =4X(1-8X+11X^2)
$$
to deduce that the coefficients $T(n)$ in the expansion
$$
Z=\sum_{n=0}^\infty T(n)X^n
$$
satisfy the four-term recurrence relation
\begin{align*}
(n+1)^3T(n+1) &=2(2n+1)(5n^2+5n+2)T(n)  \\
&\qquad -8n(7n^2+1)T(n-1)\\ 
& \qquad\qquad +22n(2n-1)(n-1)T(n-2).
\end{align*}
The single initial condition $T(0)=1$ is enough to start the recurrence.
It follows from the $q$-expansions for~$X$
and~$Z$ that $T(n)$ is always an integer.  
The first few terms $\left\{T(n)\right\}_{n=0}^{10}$ are given by
$$
1, 4, 28, 268, 3004, 36784, 476476, 6418192, 88986172, 1261473136, 18200713168.
$$
Numerical computations suggest:
\begin{conjecture}
The sequence $\left\{T(n)\right\}$ satisfies a Lucas congruence for every prime~$p$.

\noindent
That is, for every prime $p$ and any positive integer $n$,
$$
T(n) \equiv T(n_0)T(n_1)\cdots T(n_r) \pmod p
$$
where $n=n_0+n_1p+\cdots +n_rp^r$ is the $p$-adic expansion of $n$ in base $p$.
\end{conjecture}
The conjecture has been checked for all primes $p<1000$ and all $n<50000$.
One way to approach the conjecture would be to find a formula for $T(n)$ as a sum of binomial coefficients
and then apply the methods of Malik and Straub~\cite{malikstraub}. At present, such a binomial sum is
unknown.

%---------------------------------------------------

\section{The level 13 theory}
\label{Level13}
The entries in~Tables~\ref{1to35} and~\ref{1to35b} for level~13 are slightly different from, and can be deduced from, the
results in~\cite{cooperye13}.
In that reference the underlying functions, which we shall denote by $X^\ast$ and $Z^\ast$, are defined by
\begin{equation}
\label{Xast}
X^\ast = \frac{w}{1+6w+13w^2} \quad\text{and}\quad Z^\ast = \frac1{12}\left(13P(q^{13})-P(q)\right) 
\end{equation}
where $w=\eta_{13}^2/\eta_1^2$.
By~\cite[Th. 8.1]{cooperye13} we can deduce the alternative formula
\begin{equation}
\label{Zast}
Z^\ast = \eta_1^2\eta_{13}^2\times \frac{(1-X^\ast)^{2/3}}{(X^\ast)^{7/6}}
\end{equation}
which resembles the other formulas in the column for~$Z$ in Table~\ref{1to35} apart from the factor of $(1-X^\ast)^{2/3}$.
Further, by~\cite[Theorems 6.3 and 7.2]{cooperye13} it follows that the differentiation formula~\eqref{qdXdq1} and the differential equation~\eqref{D2Z} hold with
\begin{equation}
\label{Hast}
\left(B^\ast\right)^2 = 1-12X^\ast-{16X^\ast}^2\quad\text{and}\quad
H^\ast = \frac{2X^\ast(1+5X^\ast-3{X^\ast}^2+3{X^\ast}^3)}{(1-X^\ast)^2}.
\end{equation}
Observe that $H^\ast$ is not a polynomial but is a rational function instead.
The coefficients~$T^\ast(n)$ in the expansion
$$
Z^\ast = \sum_{n=0}^\infty T^\ast(n) {X^\ast}^n
$$
are integers and were shown in~\cite{cooperye} to satisfy a six-term recurrence relation.
By contrast, the data corresponding to level~13 in Tables~\ref{1to35} and~\ref{1to35b} is for the functions
\begin{equation}
\label{13definitions}
X=\frac{w}{1+5w+13w^2}\quad\text{and}\quad Z=\frac{\eta_1^2\eta_{13}^2}{X^{7/6}}
\end{equation}
with $w=\eta_{13}^2/\eta_1^2$ as before. On comparing~\eqref{13definitions} with~\eqref{Xast} and~\eqref{Zast}
we readily obtain the change of variable formulas
$$
X=\frac{X^\ast}{1-X^\ast}\quad\text{and}\quad Z=(1-X^\ast)^{1/2}Z^\ast.
$$
On making the change of variable and applying~\eqref{Xast}--\eqref{Hast}, it follows that~$X$ and~$Z$ satisfy
the differentiation formula~\eqref{qdXdq1} and the differential equation~\eqref{D2Z} with
$B^2$ and $H$ given by
$$
B^2=(1+X)(1-10X-27X^2) \quad\text{and}\quad H=\frac{X}{8}\left(12+175X+231X^2\right).
$$
It then follows by Theorem~\ref{interesting} that the coefficients $T(n)$ in the power series expansion
$$
Z=\sum_{n=0}^\infty T(n)X^n
$$
satisfy the four-term recurrence relation
\begin{align*}
(n+1)^3T(n+1) &= \frac{3}{2}(2n + 1)(3n^2 + 3n + 1)T(n) \\
&\qquad+ \frac{n}{4}(148n^2 + 27)T(n-1) \\
&\qquad\qquad+ \frac38(6n - 5)(6n - 1)(2n - 1)T(n-2)
\end{align*}
with initial conditions $T(0)=1$, $T(-1)=0$, $T(-2)=0$.
However, the first few terms are given by
$$
Z= 1 + \frac32X + \frac{91}{8}X^2 + \frac{1287}{16}X^3 + \frac{86931}{128}X^4 + \frac{1566669}{256}X^5 + 
\frac{59494799}{1024}X^6 +\cdots,
$$
and the terms $T(n)$ are not integers. We claim that $4^nT(n)$ is always an integer.
To show this, first recall from~\cite[Lemma~5.4]{cooperye13} that
\begin{equation}
\label{lemma5point4}
\eta_1^{24} = U^6 \times \frac{w}{(1+5w+13w^2)^4}\quad\text{and}\quad
\eta_{13}^{24} = U^6 \times \frac{w^{13}}{(1+5w+13w^2)^4},
\end{equation}
where
$$
U=1-\sum_{j=1}^\infty \left(\frac{j}{13}\right)\frac{jq^j}{1-q^j},
$$
and $\left(\frac{j}{13}\right)$ is the Legendre symbol.
On using~\eqref{lemma5point4} in~\eqref{13definitions} we obtain
\begin{equation}
\label{Zsquared}
Z^2 = U^2 \times (1+5w+13w^2)
\end{equation}
and it follows from the definitions of $U$ and $w$
that the coefficients in the $q$-expansion of $Z^2$ are integers.
Next, since
$$
X=\frac{w}{1+5w+13w^2},\quad\text{with}\quad w=q\prod_{j=1}^\infty \frac{(1-q^{13j})^2}{(1-q^j)^2},
$$
it follows that the coefficients in the $q$-expansion of $X$ are integers.
Further, it is clear from the above that
$$
Z^2 = 1+O(q) \quad\text{and}\quad X=q+O(q^2).
$$
Hence, the coefficients~$z(n)$ in the expansion
$$
Z^2 = \sum_{n=0}^\infty z(n)X^n = 1+3X+25X^2+195X^3+1729X^4+16107X^5+156481X^6+\cdots
$$
are integers.
Now write
$$
Z^2 = 1+I \quad\text{where}\quad I=\sum_{n=1}^\infty z(n)X^n = 3X+25X^2+195X^3+1729X^4+\cdots
$$
so that $I$ is a power series with integer coefficients and no constant term. 
By the binomial theorem we have
$$
Z = Z^2\times Z^{-1} = Z^2 \times (1+I)^{-1/2} = 
Z^2 \times \sum_{j=0}^\infty {2j \choose j} \left(-\frac{I}{4}\right)^j.
$$
When the expression on the right hand side is expanded in powers of~$X$,
the only terms in the sum that contribute to the coefficient of $X^n$ are those
$Z^2 {2j \choose j}(-\frac{I}4)^j$ with $0\leq j \leq n$. Since $Z^2$ and $I$ can each be expanded in power series in~$X$ with
integer coefficients, it follows that $4^nT(n)$ is an integer.

Let $S(n) = 4^nT(n)$. Then all of the above can be summarised as follows.
\begin{theorem}
Let $\left\{S(n)\right\}$ be the sequence defined by the recurrence relation
\begin{align*}
(n+1)^3S(n+1) &= 6(2n + 1)(3n^2 + 3n + 1)S(n) \\
&\qquad+ 4n(148n^2 + 27)S(n-1) \\
&\qquad\qquad+ 24(6n - 5)(6n - 1)(2n - 1)S(n-2)
\end{align*}
and initial conditions $S(0)=1$, $S(-1)=0$, $S(-2)=0$.
Let $X$ and $Z$ be defined by~\eqref{13definitions}, i.e.,
$$
X=\frac{w}{1+5w+13w^2}\quad\text{and}\quad Z=\frac{\eta_1^2\eta_{13}^2}{X^{7/6}},
$$
with $w=\eta_{13}^2/\eta_1^2$.  Then $S(n)$ is always an integer, and
$$
Z = \sum_{n=0}^\infty S(n) \left(\frac{X}{4}\right)^n.
$$
\end{theorem}
The first few terms $\left\{S(n)\right\}_{n=0}^{10}$ are given by:
\begin{align*}
&1,\;6,\; 182,\; 5148,\; 173862,\; 6266676,\; 237979196,\; 9366227832,\; 378768328198, \\
&15643121895492,\; 657035290739412.
\end{align*}
I did not find any Lucas property satisfied by the sequence $\left\{S(n)\right\}$. It is an open question
as to whether this sequence has any interesting number theoretic properties.

%---------------------------------------------------

\section{Irrational sequences arising from the theory of level 14}
\label{14special}
Ap{\'e}ry-like sequences that arise from the level 14 theory have been studied in \cite[Sec.~7]{alchemy},~\cite{cooperye}
and~\cite[Ex. 6]{guilleratranslation}. In this section we present the detailed modular background.
Four sequences emerge, including two irrational sequences which are new.

Let $w_{+2}$, $w_{+7}$ and $w_{+14}$ be the modular functions defined by
\begin{equation}
\label{uvwdef1}
w_{+2}=\frac{\eta_7^4\eta_{14}^4}{\eta_1^4\eta_2^4},\quad
w_{+7}=\frac{\eta_2^3\eta_{14}^3}{\eta_1^3\eta_7^3},\quad
w_{+14}=\frac{\eta_1^4\eta_{14}^4}{\eta_2^4\eta_7^4}.
\end{equation}
The notation $w_{+m}$, where $m\in\left\{2,7,14\right\}$, is an abbreviation for $w_{14+m}$
which denotes a fixed Hauptmodul for $\Gamma_0(14)+m$.
Let $y_{+2}$, $y_{+7}$ and $y_{+14}$ be the corresponding weight two modular forms defined by
$$
y_{+2}=\frac{\eta_1^5\eta_2^5}{\eta_7^3\eta_{14}^3},\quad
y_{+7}=\frac{\eta_1^4\eta_7^4}{\eta_2^2\eta_{14}^2}, \quad
y_{+14}=\frac{\eta_2^5\eta_7^5}{\eta_1^3\eta_{14}^3}.
$$
Observe that
\begin{equation}
\label{wabc1}
w_{+2}y_{+2}=w_{+7}y_{+7}=w_{+14}y_{+14}=\eta_1\eta_2\eta_7\eta_{14}.
\end{equation}
The $q$-expansions
$$
w_{+2} = q^3 + O(q^4),\quad w_{+7}=q+O(q^2)\quad\mbox{and}\quad w_{+14}=q+O(q^2)
$$
and
$$
y_{+2}=\frac{1}{q^2}+O(1),\quad y_{+7}=1+O(q)\quad\mbox{and}\quad y_{+14}=1+O(q),
$$
indicate that the pair of
functions $w_{+2}$ and $y_{+2}$ may have different properties from the pairs $w_{+7}$, $y_{+7}$ and
$w_{+14}$, $y_{+14}$. 
Indeed, the functions $w_{+7}$ and $w_{+14}$ are related by~\cite[Thm. 3.3]{cooperye}
\begin{equation}
\label{wbwc1}
\frac{1}{w_{+7}}+8w_{+7}=\frac{1}{w_{+14}}+w_{+14}-7,
\end{equation}
whereas $w_{+2}$ is related to $w_{+14}$ by the more complicated identity~\cite[Thm. 3.4]{cooperye}
\begin{equation}
\label{wawb1}
\frac{1}{w_{+2}}+2401w_{+2}
=\frac{1}{w_{+14}^3}-\frac{16}{w_{+14}^2}+\frac{48}{w_{+14}}+32+48w_{+14}-16w_{+14}^2+w_{+14}^3.
\end{equation}
The identity~\eqref{wbwc1}
could have been written in the form
\begin{equation}
\label{14sq}
\frac{1}{w_{+7}}+8w_{+7}+\epsilon=\frac{1}{w_{+14}}+w_{+14}-7+\epsilon,
\end{equation}
for any constant $\epsilon$.
This suggests considering the functions $X_\epsilon$ and $Z_\epsilon$ defined by
\begin{equation}
\label{Xdef1}
\frac{1}{X_\epsilon}=\frac{1}{w_{+7}}+8w_{+7}+\epsilon
=\frac{1}{w_{+14}}+w_{+14}-7+\epsilon
\end{equation}
and
\begin{equation}
\label{Zdef14}
Z_\epsilon = \frac{\eta_1\eta_2\eta_7\eta_{14}}{X_\epsilon},
\end{equation}
so that
$$
X_\epsilon = \frac{1}{q}+(\epsilon-3)+11q+20q^2+ 57q^3+ 92q^4 + 207q^5 + O(q^6)
$$
and
\begin{equation}
\label{Zdef1}
Z_\epsilon=y_{+7} \times (1+\epsilon w_{+7}+8w_{+7}^2) =y_{+14} \times (1+(\epsilon-7) w_{+14}+w_{+14}^2).
\end{equation}
The next result suggests that certain values of $\epsilon$ stand out as special.
\begin{theorem}
\label{thm14}
Let $B_\epsilon$ and $H_\epsilon$ be defined by
$$
B_\epsilon = \Big(1-(\epsilon - 9) X_\epsilon\Big)^{1/2}
 \Big(1-(\epsilon - 5) X_\epsilon\Big)^{1/2}
 \Big(1-2\epsilon X_\epsilon+(\epsilon^2-32)X_\epsilon^2\Big)^{1/2}
$$
and
\begin{align*}
H_\epsilon &=X_\epsilon \Bigg( \epsilon -4\,-\,\frac12\left( 7\,{\epsilon }^{2}-50\,\epsilon +24 \right) X_\epsilon
+ \left( 4\,{\epsilon }^{3}-42\,{\epsilon }^{2}+26\,\epsilon +448 \right) {X_\epsilon}^{2} \\
&\qquad\qquad\qquad\qquad\qquad\qquad\qquad\qquad
-\frac32\, \left( \epsilon -9 \right)\left( \epsilon -5 \right)  \left( {\epsilon }^{2}-32 \right) {X_\epsilon}^{3} \Bigg).
\end{align*}
Then the differentiation formula~\eqref{qdXdq1} and the differential equation~\eqref{D2Z} hold
for $B=B_\epsilon$, $H=H_\epsilon$, $X=X_\epsilon$ and $Z=Z_\epsilon$.
\end{theorem}
\begin{proof}
The results were proved for the case $\epsilon=9$ in \cite[Lemma~4.1 and Theorem~4.2]{cooperye}.
The proofs in the other cases follow by making the change of variables
$$
\frac{1}{X_\epsilon}-\epsilon = \frac{1}{X_{\epsilon=9}}-9\quad \mbox{and}\quad X_\epsilon Z_\epsilon = X_{\epsilon=9} Z_{\epsilon=9}.
$$
\end{proof}
Now consider the expansion
$$
Z_\epsilon=\sum_{n=0}^\infty T_\epsilon(n)X_\epsilon^n.
$$
By Theorem~\ref{interesting}, the coefficients $T_\epsilon(n)$ satisfy the five-term recurrence relation
\begin{align}
(n+1)^3 T_\epsilon(n+1)
&=(2n+1)((2\epsilon-7)(n^2+n) + \epsilon  - 4)T_\epsilon(n)  \label{14rec} \\
&-n((6\epsilon^2- 42\epsilon+13)n^2 + \epsilon^2 - 8\epsilon + 11)T_\epsilon(n-1) \nonumber \\
&+n(2n - 1)(n - 1)(2\epsilon^3 - 21\epsilon^2 + 13\epsilon + 224)T_\epsilon(n-2) \nonumber \\
&-n(n - 1)(n - 2)(\epsilon - 5)(\epsilon - 9)(\epsilon^2 - 32)T_\epsilon(n-3). \nonumber
\end{align}
The single initial condition $T_\epsilon(0)=1$ is enough to start the recurrence.
\\
If \mbox{$\epsilon \in \left\{5,\, 9, \, \pm \sqrt{32}\right\}$}
then the degree of $H_\epsilon$ drops by one, the function $B_\epsilon$ also simplifies, and
the recurrence relation~\eqref{14rec} reduces to a four-term relation of the form
\begin{align}
(n+1)^3T_\epsilon(n+1) =(2n+1)(an^2+an+b)&T_\epsilon(n)  +n(cn^2+d)T_\epsilon(n-1) \label{4term} \\
& +en(2n-1)(n-1)T_\epsilon(n-2). \nonumber 
\end{align}
It may be noted that the values $\epsilon=5$, $9$, $\sqrt{32}$ and $-\sqrt{32}$ are precisely the values that
make one of the sides of~\eqref{14sq} a square, and we will write $T_{14A}(n)$,  $T_{14B}(n)$,  $T_{14C}(n)$
and $T_{14\overline{C}}(n)$, respectively, for $T_\epsilon(n)$ in these cases.

The sequence $\left\{T_{14B}(n)\right\}$ was first studied in~\cite{guilleratranslation}
and investigated further in~\cite{cooperye}. The sequence $\left\{T_{14A}(n)\right\}$ seems to have first arisen in~\cite{alchemy} in the course
of proving some conjectures of Z.-W. Sun~\cite{sun}. The sequences $\left\{T_{14A}(n)\right\}$ and $\left\{T_{14B}(n)\right\}$
are listed in Table~\ref{1to35} along with the corresponding modular forms.
The sequences $\left\{T_{14C}(n)\right\}$ and $\left\{T_{14\overline{C}}(n)\right\}$ are new.

The first few terms of the sequences $\left\{T_{14A}(n)\right\}$, $\left\{T_{14B}(n)\right\}$
and $\left\{T_{14C}(n)\right\}$ are listed in Table~\ref{table14}.
If 
$$T_{14C}(n) = u+v\sqrt{2}$$
for some integers $u$ and $v$ then
$$T_{14\overline{C}}(n) = u-v\sqrt{2},$$
so that $T_{14\overline{C}}(n)$
is obtained from $T_{14C}(n)$ by conjugation.

The values of the constants $a$, $b$, $c$, $d$ and $e$ corresponding to the
recurrence relation~\eqref{4term} for the values $\epsilon=5$, $9$, $\pm \sqrt{32}$ are given in Table~\ref{table5}
under the heading level $14$.

The following formulas, involving binomial sums, were proved in~\cite{alchemy}:
\begin{align*}
T_{14A}(n)
&= \sum_{j,k} {n+j \choose 2j+2k} { 2j+2k \choose j+k} {2k \choose k}^2 {k \choose j} \\
&= \sum_{j,k} (-1)^{n-j}{n+j \choose 2j+2k} {2j+2k \choose j+k} {2k \choose k} {j+2k \choose k} {j+k \choose k} \\
&= \sum_{j,k} {n+j-k \choose 2j+2k} {2j+2k \choose j+k} {j+k \choose k}^2 {2j \choose j+k} (-3)^{n-j-3k}.
\end{align*}
From~\eqref{Zdef14} we have
$$
Z_\epsilon X_\epsilon = \eta_1\eta_2\eta_7\eta_{14}
$$
is independent of $\epsilon$, 
and on using~\eqref{Xdef1} it follows further that
$$
\sum_{n=0}^\infty T_\epsilon(n) \left(\frac{w}{1+\epsilon w+8w^2}\right)^{n+1}
= w-4w^2+16w^3-72w^4+368w^5-2080w^6+\cdots
$$
is also independent of $\epsilon$.
In particular, the generating functions for the four sequences $\left\{T_{\epsilon}(n)\right\}$ for
$\epsilon\in\left\{5,9,\pm\sqrt{32}\right\}$ are related as follows.
\begin{theorem}
The following identity holds, in a neighborhood of $w=0$:
\begin{align*}
\sum_{n=0}^\infty T_{14A}(n)\left(\frac{w}{1+5w+8w^2}\right)^{n+1}
&=\sum_{n=0}^\infty T_{14B}(n)\left(\frac{w}{1+9w+8w^2}\right)^{n+1} \\
&=\sum_{n=0}^\infty T_{14C}(n)\left(\frac{w}{1+\sqrt{32}w+8w^2}\right)^{n+1} \\
&=\sum_{n=0}^\infty T_{14\overline{C}}(n)\left(\frac{w}{1-\sqrt{32}w+8w^2}\right)^{n+1}.
\end{align*}
\end{theorem}
\begin{table}
\caption{The first few terms of the level $14$ sequences $\left\{T_\epsilon(n)\right\}$}
\label{table14}
{\renewcommand{\arraystretch}{1.4}
{\begin{tabular}{|c|c|c|c|}
\hline
$n$ & $T_{14A}(n): \; \epsilon=5$ & $T_{14B}(n): \epsilon=9$ & $T_{14C}(n): \; \epsilon=\sqrt{32}$  \\
\hline
0 & 1 & $1$ & $1$  \\ 
1 & 1 & $5$ & $-4+4\sqrt{2}$  \\  
2 & 9 &  $33$ &$56-32\sqrt{2}$  \\ 
3 & 49 & $269$ & $-520+416\sqrt{2}$  \\ 
4 & 385 & $2545$ & $6512-4224\sqrt{2}$  \\ 
5 & 2961 & $26565$ &  $-69664+52416\sqrt{2}$  \\ 
6 & 24801 & $295785$ & $862904-582400\sqrt{2}$  \\ 
7 & 212409 & $3441765$ & $-9870928+7232544\sqrt{2}$  \\ 
8 & 1878129 & $41336145$ &  $123164432-84724224\sqrt{2}$  \\ 
9 & 16924945 & $508419125$ &  $-1472036416+1063509568\sqrt{2}$  \\ 
10 & $\;\;155204329\;\;$ & $\;\;6370849633\;\;$ &  $\;\;18601926816-12933544448\sqrt{2}\;\;$  \\
\hline
\end{tabular}}}
\end{table}

We end the section with a conjectural congruence identity that has been tested extensively by computer.

\begin{conjecture}
\label{conj14}
The sequences $\left\{T_{14A}(n)\right\}$ and $\left\{T_{14B}(n)\right\}$ satisfy the Lucas
congruence~\eqref{a4} for every prime $p$. The sequences $\left\{T_{14C}(n)\right\}$
and $\left\{T_{14\overline{C}}(n)\right\}$ satisfy the Lucas congruence for
a prime $p$ if and only if $p=2$ or $p$ is congruent to $1$ or $7$ modulo $8$.
\end{conjecture}

%---------------------------------------------------

\section{Complex valued sequences arising from the level 15 theory}
\label{154}
Despite having many similarities, the level 15 theory has only been studied in~\cite{cooperye}. For example, in the last
section of~\cite{alchemy}, it was noted with surprise that while the level 14 theory occurred several times
in that work, the level 15 theory did not show up at all. In this section we develop the modular background
and obtain four sequences, three of which are new.

Using similar notation as for the previous section, let $w_{+3}$, $w_{+5}$ and $w_{+15}$ be the modular functions defined by
\begin{equation}
\label{uvwdef}
w_{+3}=\frac{\eta_5^3\eta_{15}^3}{\eta_1^3\eta_3^3},\quad
w_{+5}=\frac{\eta_3^2\eta_{15}^2}{\eta_1^2\eta_5^2},\quad
w_{+15}=\frac{\eta_1^3\eta_{15}^3}{\eta_3^3\eta_5^3},
\end{equation}
and let $y_{+3}$, $y_{+5}$ and $y_{+15}$ be the corresponding weight two modular forms defined by
$$
y_{+3}=\frac{\eta_1^4\eta_3^4}{\eta_5^2\eta_{15}^2},\quad
y_{+5}=\frac{\eta_1^3\eta_5^3}{\eta_3\eta_{15}}, \quad
y_{+15}=\frac{\eta_3^4\eta_5^4}{\eta_1^2\eta_{15}^2}.
$$
Analogous to \eqref{wabc1} we have
\begin{equation}
\label{wabc}
w_{+3}y_{+3}=w_{+5}y_{+5}=w_{+15}y_{+15}=\eta_1\eta_3\eta_5\eta_{15}.
\end{equation}
Proceeding as in the previous section, the $q$-expansions
$$
w_{+3} = q^2 + O(q^3),\quad w_{+5}=q+O(q^2)\quad\mbox{and}\quad w_{+15}=q+O(q^2)
$$
and
$$
y_{+3}=\frac{1}{q}+O(1),\quad y_{+5}=1+O(q)\quad\mbox{and}\quad y_{+15}=1+O(q),
$$
suggest that the pair of
functions $w_{+3}$ and $y_{+3}$ may have different properties from the pairs $w_{+5}$, $y_{+5}$ and
$w_{+15}$, $y_{+15}$. 
The functions $w_{+5}$ and $w_{+15}$ are related by \cite[Thm. 5.3]{cooperye}
\begin{equation}
\label{wbwc}
\frac{1}{w_{+5}}+9w_{+5}+5=\frac{1}{w_{+15}}-w_{+15},
\end{equation}
whereas $w_{+3}$ is related to $w_{+5}$ by the more complicated identity \cite[Thm. 5.4]{cooperye}
\begin{equation}
\label{wawb}
\frac{1}{w_{+3}}-125w_{+3}
=\frac{1}{w_{+5}^2}+\frac{1}{w_{+5}}-9w_{+5}-81w_{+5}^2.
\end{equation}
The identity~\eqref{wbwc}
could have been written in the form
\begin{equation}
\label{15sq}
\frac{1}{w_{+5}}+9w_{+5}+5+\epsilon=\frac{1}{w_{+15}}-w_{+15}+\epsilon
\end{equation}
for any constant $\epsilon$.
This suggests considering the functions $X_\epsilon$ and $Z_\epsilon$ defined by
\begin{equation}
\label{Xdef}
\frac{1}{X_\epsilon}=\frac{1}{w_{+15}}-w_{+15}+\epsilon\quad\mbox{and}\quad X_\epsilon Z_\epsilon = \eta_1\eta_3\eta_5\eta_{15}
\end{equation}
so that
\begin{equation}
\label{Zdef}
Z_\epsilon=y_{+15} \times (1+\epsilon w_{+15}-w_{+15}^2).
\end{equation}
To help determine interesting values of $\epsilon$, we have the following analogue of Theorem~\ref{thm14}.
\begin{theorem}
\label{t15}
Let $B_\epsilon$ and $H_\epsilon$ be defined by
$$
B_\epsilon = \Big(1-(\epsilon - 1) X_\epsilon\Big)^{1/2}
 \Big(1-(\epsilon + 11) X_\epsilon\Big)^{1/2}
 \Big(1-2\epsilon X_\epsilon+(\epsilon^2 + 4)X_\epsilon^2\Big)^{1/2}
$$
and
\begin{align*}
H_\epsilon &=X_\epsilon \Bigg( \epsilon +2\,-\,\frac12\left( 7\,{\epsilon }^{2}+34\,\epsilon -8 \right) X_\epsilon
+ \left( 4\,{\epsilon }^{3}+30\,{\epsilon }^{2}-14\,\epsilon +40 \right) {X_\epsilon}^{2} \\
&\qquad\qquad\qquad\qquad\qquad\qquad\qquad\qquad
-\frac32\, \left( \epsilon -1 \right)\left( \epsilon +11 \right)  \left( {\epsilon }^{2}+4 \right) {X_\epsilon}^{3} \Bigg).
\end{align*}
Then the differentiation formula~\eqref{qdXdq1} and the differential equation~\eqref{D2Z} hold
for $B=B_\epsilon$, $H=H_\epsilon$, $X=X_\epsilon$ and $Z=Z_\epsilon$.
\end{theorem}
\begin{proof}
The identities were proved for the case $\epsilon=1$ in \cite[Lemma~6.1 and Theorem~6.2]{cooperye}.
The proofs in the other cases follow by the change of variables
$$
\frac{1}{X_\epsilon}-\epsilon = \frac{1}{X_{\epsilon=1}}-1\quad \mbox{and}\quad X_\epsilon Z_\epsilon = X_{\epsilon=1} Z_{\epsilon=1}.
$$
\end{proof}
Now consider the expansion
$$
Z_\epsilon=\sum_{n=0}^\infty T_\epsilon(n)X_\epsilon^n.
$$
By Theorem~\ref{interesting}, the coefficients $T_\epsilon(n)$ satisfy the five-term recurrence relation
\begin{align}
(n+1)^3 T_\epsilon(n+1)
&=(2n+1)((2\epsilon+5)(n^2+n) + \epsilon  +2)T_\epsilon(n)  \label{15rec} \\
&-n((6\epsilon^2+30\epsilon-7)n^2 + \epsilon^2 +4\epsilon - 1)T_\epsilon(n-1) \nonumber \\
&+n(2n - 1)(n - 1)(2\epsilon^3 +15\epsilon^2 -7\epsilon + 20)T_\epsilon(n-2) \nonumber \\
&-n(n - 1)(n - 2)(\epsilon - 1)(\epsilon +11)(\epsilon^2 +4)T_\epsilon(n-3). \nonumber
\end{align}
If $\epsilon \in \left\{1,\, -11, \, \pm 2i\right\}$
then the degree of $H_\epsilon$ drops by one, the function $B_\epsilon$ also simplifies, and
the recurrence relation~\eqref{15rec} reduces to a four-term relation of the form~\eqref{4term}.
It may be noted that the values $\epsilon=1$, $-11$, $2i$ and $-2i$ are precisely the values that
make one of the sides of~\eqref{15sq} a square.
The terms in the sequences in these cases will be denoted by $T_{15A}(n)$,  $T_{15B}(n)$,  $T_{15C}(n)$
and $T_{15\overline{C}}(n)$, respectively. The sequence~$T_{15A}(n)$ was studied in~\cite{cooperye}. The other three
sequences are (until now) unpublished; I found them in~2015 and presented them at the NTDU conference in 2015,~\cite{cooperNTDU}.

The values of the contants $a$, $b$, $c$, $d$ and $e$ are given in Table~\ref{table5}
under the heading level~15.
The first few terms in the sequences corresponding to $\epsilon\in\left\{1,-11,2i\right\}$
are listed in Table~\ref{table15}.
The case $\epsilon=-2i$ is not listed in Table~\ref{table15} as the corresponding data
will just be the complex conjugate of the data obtained for~$\epsilon=2i$.
For each sequence, the single initial condition $T_\epsilon(0)=1$ is enough to start the recurrence.
From~\eqref{Xdef} we have
$$
Z_\epsilon X_\epsilon = \eta_1\eta_3\eta_5\eta_{15}
$$
is independent of $\epsilon$, 
and on using~\eqref{Xdef1} it follows further that
$$
\sum_{n=0}^\infty T_\epsilon(n) \left(\frac{w}{1+\epsilon w-w^2}\right)^{n+1}
= w+2w^2+11w^3+72w^4+545w^5+4450w^6+\cdots
$$
is also independent of $\epsilon$.
In particular, the generating functions for the four sequences $\left\{T_{\epsilon}(n)\right\}$ for
$\epsilon\in\left\{1,-11,\pm 2i\right\}$ are related as follows.
\begin{theorem}
The following identity holds, in a neighborhood of $w=0$:
\begin{align*}
\sum_{n=0}^\infty T_{15A}(n)\left(\frac{w}{1+w-w^2}\right)^{n+1}
&=\sum_{n=0}^\infty T_{15B}(n)\left(\frac{w}{1-11w-w^2}\right)^{n+1} \\
&=\sum_{n=0}^\infty T_{15C}(n)\left(\frac{w}{1+2iw-w^2}\right)^{n+1} \\
&=\sum_{n=0}^\infty T_{15\overline{C}}(n)\left(\frac{w}{1-2iw-w^2}\right)^{n+1}.
\end{align*}
\end{theorem}

We end this section with a conjecture that is the analogue of Conjecture~\ref{conj14}.
\begin{conjecture}
The sequences $\left\{T_{15A}(n)\right\}$ and $\left\{T_{15B}(n)\right\}$ satisfy the Lucas
congruence for every prime $p$. The sequences $\left\{T_{15C}(n)\right\}$
and $\left\{T_{15\overline{C}}(n)\right\}$ satisfy the Lucas congruence for
a prime $p$ if and only if $p=2$ or $p$ is congruent to~$1$ modulo~$4$.
\end{conjecture}

\begin{table}
\caption{The first few terms of the level $15$ sequences $\left\{T_\epsilon(n)\right\}$}
\label{table15}
{\renewcommand{\arraystretch}{1.4}
{\begin{tabular}{|c|c|c|c|}
\hline
$n$ & $T_{15A}(n):\; \epsilon=1$ & $T_{15B}(n): \; \epsilon=-11$ & $T_{15C}(n): \; \epsilon=2i$  \\
\hline
0 & 1 & $1$ & $1$  \\ 
1 & 3 & $-9\;$ & $2 + 2 i$  \\  
2 & 15 &  $87$ &$6 + 8 i$  \\ 
3 & 105 & $-867\;$ & $44 + 52 i$  \\ 
4 & 855 & $8775$ & $290 + 480 i$  \\ 
5 & 7533 & $-89559\;$ &  $1612 + 4372 i$  \\ 
6 & 69909 & $918141$ & $7140 + 39568 i$  \\ 
7 & 673515 & $-9432873\;$ & $2536 + 361688 i$  \\ 
8 & 6673095 & $96984423$ &  $-559166 + 3303552 i$  \\ 
9 & 67565445 & $-997061295\;$ &  $-10693900 + 29823140 i$  \\ 
10 & $\;\;696024945\;\;$ & $\;\;10245169737\;\;$ &  $\;\;-151732284 + 264070928 i\;\;$  \\
\hline
\end{tabular}}}
\end{table}

%---------------------------------------------------

\section{Levels 20, 21 and 24, and a general conjecture}
\label{level202124}
In this section we will consider the remaining entries in Table~\ref{1to35} for which the coefficients $T(n)$ satisfy
a four-term recurrence relation, namely levels 20, 21 and~24.

First, by applying Theorem~\ref{interesting} to the data in Tables~\ref{1to35} and~\ref{1to35b} for levels 20 and 21,
we obtain the four-term recurrence relations 
\begin{align*}
(n+1)^3T(n+1)
&=4(2n+1)(2n^2+2n+1)T(n) \\
&\quad -16n(4n^2+1)T(n-1) +8(2n-1)^3T(n-2)
\intertext{and}
(n+1)^3T(n+1)
&=-(2n+1)(n^2+n+1)T(n) \\
&\quad +n(35n^2+12)T(n-1) +6(2n-1)(3n-1)(3n-2)T(n-2),
\end{align*}
respectively.
Neither of these recurrence relations are self-starting in the sense that specifying the single initial condition $T(0)=1$
is not enough to start the recurrence; in order to handle the cases $n=0$ and $n=1$ it is
also necessary to stipulate that $T(-1)=0$ and $T(-2)=0$.
Moreover, no Lucas congruences for these sequences have been found. 

On the other hand, applying Theorem~\ref{interesting} to the data in Tables~\ref{1to35} and~\ref{1to35b} for level~24
gives the recurrence relation
\begin{align}
(n+1)^3T(n+1)\label{24rr}
&=2(2n+1)(2n^2+2n+1)T(n) \\
&\quad +4n(4n^2+1)T(n-1) -64n(n-1)(2n-1)T(n-2), \nonumber
\end{align}
and the single initial condition $T(0)=1$ is enough to start the recurrence relation.
That is because the coefficient of $T(n-1)$ contains a factor of $n$, and the coefficient of $T(n-2)$ contains 
a factor of $n(n-1)$.
Furthermore, numerical evidence suggests:
\begin{conjecture}
The sequence defined by the recurrence relation~\eqref{24rr} and initial condition~$T(0)=1$
satisfies a Lucas congruence for all primes~$p$.
\end{conjecture}

Now that we have investigated all of the four-term recurrence relations in Table~\ref{1to35},
it appears that the sequences which satisfy Lucas congruences are the self-starting sequences.
By Theorem~\ref{interesting}, the sequence $T(n)$ satisfies a four-term recurrence relation
if the corresponding polynomials $B^2$ and $H$ are of degree 3, i.e.,
$$
B^2 = 1+g_1X+g_2X^2+g_3X^3\quad\text{and}\quad H=h_1X+h_2X^2+h_3X^3
$$
and the recurrence relation is given by
\begin{align*}
&(n+1)^3T(n+1) \\
&\qquad +\left(n+\mbox{$\frac12$}\right)\left(g_1n^2+g_1n-2h_1\right)T(n) \\
&\qquad\qquad + n\left(g_2n^2-g_2-2h_2\right)T(n-1) \\
&\qquad\qquad\qquad+\left(n-\mbox{$\frac12$}\right)\left(g_3n^2-g_3n-2g_3-2h_3\right)T(n-2)=0.
\end{align*}
The sequence will be self-starting provided the single initial condition $T(0)=1$ can be used to start the recurrence,
and this will be the case when $n$ divides the coefficient of $T(n-1)$ and when $n(n-1)$ divides the coefficient of $T(n-2)$.
The first condition is already satisfied, and the second condition will be satisfied if and only if $g_3=-h_3$, in which
case the recurrence relation becomes 
\begin{align*}
&(n+1)^3T(n+1) \\
&\qquad +\left(n+\mbox{$\frac12$}\right)\left(g_1n^2+g_1n-2h_1\right)T(n) \\
&\qquad\qquad + n\left(g_2n^2-g_2-2h_2\right)T(n-1) \\
&\qquad\qquad\qquad+g_3n(n-1)\left(n-\mbox{$\frac12$}\right)T(n-2)=0,
\end{align*}
or equivalently
\begin{align}
(n+1)^3T(n+1) =(2n+1)(an^2+an+b)&T(n)  +n(cn^2+d)T(n-1) \label{4tr} \\
& +en(2n-1)(n-1)T(n-2), \nonumber 
\end{align}
where 
$$
a=-\frac{g_1}{2}, \quad b=h_1,\quad c=-{g_2},\quad d=g_2+2h_2,\quad e=-\frac{g_3}{2} = \frac{h_3}{2}.
$$
Table~\ref{table5} contains the values of the parameters $(a,b,c,d,e)$ for which the sequence
defined by the recurrence relation~\eqref{4tr} and initial condition~$T(0)=1$ is conjectured
to satisfy Lucas congruences.
\begin{table}
\caption{Parameters for the four-term recurrence relation
\begin{align*} \hspace{-3cm} (n+1)^3&T(n+1) =(2n+1)(an^2+an+b)T(n)\\
& +n(cn^2+d)T(n-1) +en(2n-1)(n-1)T(n-2) \hspace{1cm}
\end{align*}
 and initial condition $T(0)=1$
that produce sequences which are conjectured to satisfy Lucas congruences.}
\vspace{0.5cm}
\label{table5}
{\renewcommand{\arraystretch}{2.4}
{\renewcommand{\arraystretch}{2.4}
{\renewcommand{\arraystretch}{2.9}
\centering
{\resizebox{12.5cm}{!}{
{\begin{tabular}{|c|c||c|c|c|c|c|}
\hline
Level & Name & $a$ & $b$ & $c$ & $d$ & $e$ \\
\hline
\hline
$11$ & $T_{11}(n)$ & $10$ & $4$ & $-56$ & $-8$ & $22$  \\
\hline\hline
 \multirow{3}{*}&$T_{14A}(n)$ & $3$ & $1$ & $47$ & $4$ & $14$ \\
\cline{2-7}
{$14$} &$T_{14B}(n)$ & $11$ & $5$  & $-121$ & $-20$ & $98$ \\
\cline{2-7}
&$T_{14C}(n)$ &  $8\sqrt{2}-7$ & $4\sqrt{2}-4$ & $168\sqrt{2}-205$ & $32\sqrt{2}-43$ & $28(11\sqrt{2}-16)$ \\
\cline{2-7}
&$T_{14\overline{C}}(n)$ &  $-8\sqrt{2}-7$ & $-4\sqrt{2}-4$ & $-168\sqrt{2}-205$ & $-32\sqrt{2}-43$ & $-28(11\sqrt{2}+16)$ \\
\hline \hline
 \multirow{3}{*}&$T_{15A}(n)$ & $7$ & $3$ & $-29$ & $-4$ & $30$ \\
\cline{2-7}
{$15$}&$T_{15B}(n)$ & $-17$ & $-9$  & $-389$ & $-76$ & $-750$ \\
\cline{2-7}
&$T_{15C}(n)$ & $5+4i$ & $2+2i$ & $31-60i$ & $5-8i$ & $-40-30i$ \\
\cline{2-7}
&$T_{15\overline{C}}(n)$ & $5-4i$ & $2-2i$ & $31+60i$ & $5+8i$ & $-40+30i$ \\
\hline \hline
$24$ &$T_{24}(n)$ &  $4$ & $2$ & $16$ & $4$ & $-64$  \\
\hline\hline
\end{tabular}}}}}}}
\end{table}

%---------------------------------------------------

\section{Asymptotic expansions}
An asymptotic formula for the Ap{\'e}ry numbers was mentioned in~\eqref{cohen}.
We seek analogous asymptotic expansions for sequences defined by the recurrence relation~\eqref{generalT} in Theorem~\ref{interesting}, viz.
\begin{equation}
\label{gT}
(n+1)\sum_{j=0}^k g_j(n+1-{\textstyle \frac{j}{2}})(n+1-j)T(n+1-j)
=2\sum_{j=1}^k h_j(n+1-{\textstyle \frac{j}{2}})T(n+1-j),
\end{equation}
where $g_j$ and $h_j$ are constants and $g_0=1$.
We start with the ansatz
\begin{equation}
\label{ansatz}
T(n) \sim C \,n^\alpha R^n \left(1+\frac{b_1}{n} + \frac{b_2}{n^2}+\cdots\right)
\end{equation}
and substitute in~\eqref{gT}. To compute the expansion, note that for $j=0,1,2,\ldots$ we have
\begin{align*}
\frac{T(n+1-j)}{n^\alpha R^n}
&= \frac{(n+1-j)^\alpha}{n^\alpha} \frac{R^{n+1-j}}{R^{n+1}} \left(1+\frac{b_1}{n+1-j}+\frac{b_2}{(n+1-j)^2}+\cdots\right) \\
&=\frac{1}{R^{j}} \left(1-\frac{\alpha(j-1)}{n}+\frac{\alpha(\alpha-1)}{2} \frac{(j-1)^2}{n^2}+\cdots\right) \\
&\qquad \qquad \times \left(1+\frac{b_1}{n}+\frac{(j-1)b_1+b_2}{n^2}+\cdots\right).
\end{align*}
Now divide~\eqref{gT} by $n^{\alpha} R^{n+1} n^2(n+1)$ and substitute the above expansion to obtain
\pagebreak[4] %%%% Used to prevent awkward page break in the middle of the following equation
\begin{align*}
&\sum_{j=0}^k \frac{g_j}{R^{j}} \left(1-\frac{j-1}{n}\right)\left(1-\frac{j-2}{2n}\right)\left(1-\frac{\alpha(j-1)}{n} + \frac{\alpha(\alpha-1)}{2}\frac{(j-1)^2}{n^2}+\cdots\right) \\
&\qquad\qquad\qquad \times \left(1+\frac{b_1}{n}+\frac{(j-1)b_1+b_2}{n^2}+\cdots\right) \\
&= 2\sum_{j=1}^k \frac{h_j}{R^{j}} \left(\frac{1}{n^2} + \cdots \right),
\end{align*}
where we have expanded as far as terms involving $n^{-2}$.

Equating coefficients of $n^0$ gives
$$
\sum_{j=0}^k \frac{g_j}{R^{j}}=0.
$$
Hence, $R^{-1}$ is a root of the polynomial $\displaystyle{G(x) = 1+\sum_{j=1}^k g_j x^j}$. 

Next, equating coefficients of $n^{-1}$ gives
$$
\sum_{j=0}^k \frac{g_j}{R^{j}} \left( -(j-1)-\frac{(j-1)}{2} - \alpha(j-1)\right) + b_1\sum_{j=0}^k \frac{g_j}{R^{j}}=0.
$$
On using the fact that $R^{-1}$ is a root of $G(x)$ and simplifying, we obtain
$$
\sum_{j=0}^k \frac{g_j}{R^{j}} \left(-\frac{3j}{2}\right) = \alpha \sum_{j=0}^k \frac{jg_j}{R^{j}},\quad\text{i.e.,}\quad -\frac32G'(R^{-1}) = \alpha G'(R^{-1}).
$$
If we assume further that $R^{-1}$ is a simple root of $G$, then it follows that $\alpha = -3/2$.

Setting $\alpha = -3/2$ and equating coefficients of $n^{-2}$ gives
$$
\sum_{j=0}^k \frac{g_j}{R^{j}} \left(\frac{j^2}{8} + b_1j-\frac{b_1}{2}+b_2-\frac18\right) = 2\sum_{j=1}^k \frac{h_j}{R^{j}}.
$$
Once again we use the fact that $R^{-1}$ a root of $G(x)$ to simplify and obtain
$$
b_1\sum_{j=0}^k \frac{jg_j}{R^j} + \frac18\sum_{j=0}^k \frac{j^2g_j}{R^j} = 2\sum_{j=1}^k \frac{h_j}{R^j},
$$
and this can be solved to give
$$
b_1 = \frac{16H(R^{-1}) - G^{\ast\ast}(R^{-1})}{8G^\ast(R^{-1})},
$$
where
\begin{equation}
\label{ast}
H(x) = \sum_{j=1}^k h_j x^j, \quad G^\ast(x) = x\frac{\ud}{\ud x}G(x) \quad \text{and}\quad G^{\ast\ast}(x) = x\frac{\ud}{\ud x}\left(x \frac{\ud}{\ud x}G(x)\right).
\end{equation}
If needed, additional coefficients $b_2$, $b_3,\ldots$ could be computed by expanding further and equating coefficients of $n^{-3}$, $n^{-4}, \ldots$.

It is well-known that the value of $C$ in~\eqref{ansatz} cannot be determined solely from the recurrence relation.
Instead, forward differences can be used to determine the value of~$C$ accurately, e.g., see~\cite[p. 954]{zagierasymptotic} for a description of the method.
Then the PSLQ algorithm can be used to conjecture an exact value of the constant.
Alternatively, if additional facts are known about the sequence, for example if
a representation as a binomial sum is available, then it may be possible to compute the value of~$C$ analytically, e.g., see Hirschhorn~\cite{mdh1}, \cite{mdh3}.
In principle, Hirschhorn's approach could be applied to the sequences $T_{14A}$ and $T_{24}$, but binomial sums are not known for the other sequences in Table~\ref{table5} so
the values of~$C$ are determined numerically and hence are conjectural.

The calculations and discussions above can be used in conjunction with Theorem~\ref{interesting} and summarised as follows.
\begin{theorem}
\label{generalasymptotic}
Suppose $Z=Z(X)$ is analytic in a neighbourhood of $X=0$ with $Z(0)=Z_0 \neq 0$, and satisfies the second order non-linear
differential equation
$$
\uD^2Z - \frac{1}{2Z}\left(\uD Z\right)^2=HZ,
$$
where $\uD$ is the differential operator 
$$
\uD = BX \frac{\ud}{\ud X},
$$
where
$$
B=\sqrt{G}\quad\text{and}\quad G = 1+\sum_{j=1}^k g_jX^j,
$$
and
$$
H=\sum_{j=1}^k h_jX^j.
$$
Suppose $G(x)$ has a unique root $r_0$ of smallest modulus and assume that $r_0$ is simple.
\\
Then the coefficients $T(n)$ in the power series expansion
$$
Z=\sum_{n=0}^\infty T(n)X^n
$$
have an asymptotic expansion given by
\begin{equation}
\label{ansatzA}
T(n) \sim C \, n^{-3/2}\, R^n\left(1+\frac{b_1}{n} + O\left(\frac{1}{n^2}\right)\right)
\end{equation}
where $R = 1/r_0$ and $\displaystyle{b_1 = \frac{16H(r_0) - G^{\ast\ast}(r_0)}{8G^\ast(r_0)}},$ and $G^{\ast}$ and $G^{\ast\ast}$ are defined by~\eqref{ast}.
\end{theorem}
It is useful to do some checks of this formula with known results. By the discussion preceding and following Theorem~\ref{ASZ} we have corresponding to the Ap{\'e}ry numbers
$$
G(x) = 1+2\alpha x +(\alpha^2+4\gamma)x^2 \quad\text{and}\quad H(x) = (2\beta-\alpha)x - \frac{(\alpha^2+4\gamma)x^2}{2}
$$
with $(\alpha,\beta,\gamma) = (-17,-6,-72)$.
Solving $G(x)=0$ gives smallest root \mbox{$r_0= 17-12\sqrt{2}$} and hence
$$
R = \frac{1}{r_0} = 17+12\sqrt{2} = (1+\sqrt{2})^4.
$$
Furthermore, 
$$
b_1= \frac{16H(r_0) - G^{\ast\ast}(r_0)}{8G^\ast(r_0)}= \left.\frac{57-6x}{8x-136}\right|_{x=17-12\sqrt{2}} = -\frac34+\frac{15\sqrt{2}}{64}.
$$
Hence we obtain for the Ap{\'e}ry numbers 
$$
A(n) \sim \frac{C}{n^{3/2}} (1+2\sqrt{2})^4 \left(1+ \frac{15\sqrt{2}-48}{64n}+O\left(\frac{1}{n^2}\right)\right).
$$
The value of $C$ may be determined numerically and we find that $C=\frac{(1+2\sqrt{2})^2}{2^{9/4}\pi^{3/2}}$. These results are in agreement with Hirschhorn~\cite{mdh1}
and \cite{mdh2}, where coefficients of $1/n^2$ and $1/n^3$ are also determined.

As another example, consider the case of level~7. From Table~\ref{1to35b} we have that \mbox{$G(x) = (1+x)(1-27x)$} and $H(x) = 4x(1+3x)$.
The root of $G$ of smallest modulus is $r_0=1/27$ and therefore $R=27$. Moreover,
$$
b_1= \frac{16H(r_0) - G^{\ast\ast}(r_0)}{8G^\ast(r_0)}= \left.\frac{-45-150x}{104+216x}\right|_{x=1/27} = -\frac{65}{144}.
$$
Hence we obtain
$$
T(n) \sim \frac{C}{n^{3/2}} 27^n \left(1- \frac{65}{144n}+O\left(\frac{1}{n^2}\right)\right).
$$
Once again, the value of $C$ may be determined numerically and we find that \mbox{$C=\frac{\sqrt{27}}{4\pi^{3/2}}$}, and all of this is in agreement with
another result of Hirschhorn~\cite{mdh3}.

Theorem~\ref{generalasymptotic} can be used to determine the parameters $R$ and $b_1$ in the asymptotic expansions for the eight sequences in Table~\ref{table5}.
The polynomials $G$ and $H$  for the sequences $T_{11}$, $T_{14A}$, $T_{14B}$, $T_{15A}$, $T_{15B}$, and $T_{24}$ can be found in Table~\ref{1to35b}, and the polynomials for the
other sequences can be extracted from Theorems~\ref{thm14} and~\ref{t15}. The corresponding parameter~$C$ in each asymptotic expansion is determined numerically using finite differences.
The results are summarised in Table~\ref{table7}.
The result in the table for $T_{14B}(n)$ agrees with the formula given by Guillera and Zudilin~\cite[Ex. 6]{guilleratranslation}.

The point to note is that the dominant term in the asymptotic expansion can be determined, up to the constant~$C$, simply by inspecting the polynomials $G=B^2$ in Table~\ref{1to35b}.
And if a finer approximation such as the $1/n$ term is required, then the polynomial $H$ in Table~\ref{1to35b} comes into play.

\begin{table}
\caption{Parameters in the asymptotic expansion \\ $\displaystyle{T(n) \sim C \, \frac{ R^n}{n^{3/2}} \left(1+\frac{b_1}{n} + O\left(\frac{1}{n^2}\right)\right).}$} \label{table7}
{\renewcommand{\arraystretch}{2.4}
{\resizebox{12.5cm}{!}{
{\begin{tabular}{|c||c|c|c|}
\hline
Sequence & $R$ & $b_1$ & Conjectured value of $C$  \\
\hline
\multirow{2}{1.1cm}{$T_{11}(n)$} & $\displaystyle{1-\frac{20}{R}+\frac{56}{R^2}-\frac{44}{R^3}=0}$ 
& $\displaystyle{-\frac18\left(\frac{275R^{-2} - 184R^{-1} + 21}{33R^{-2} - 28R^{-1} + 5}\right)}$ & $\displaystyle \frac{xR}{\pi^{3/2}} = {0.3287368\ldots,\;}$ where   \\
 & so $\;R=16.82750\ldots$  & $=-0.3995791\cdots$ &$2^{10}\cdot 11 x^6 + 2^5\cdot 29x^2-11=0$ \\[1ex]
\hline
$T_{14A}(n),\;\epsilon=5$  & $5+4\sqrt{2}$ & $\qquad\displaystyle{-\frac{223}{196}+\frac{1025\sqrt{2}}{3136}}\qquad$ & $\displaystyle{\frac{5+4\sqrt{2}}{4\pi^{3/2}} \times \left(\frac{9\sqrt{2}-8}{14}\right)^{1/2}}$ \\
 $T_{14B}(n),\;\epsilon=9$ & $(1+2\sqrt{2})^2$  & $\displaystyle{-\frac{7}{4}+\frac{69\sqrt{2}}{64}}$ & $\displaystyle{\frac{(1+2\sqrt{2})^2}{4\pi^{3/2}} \times \left(\frac{8-5\sqrt{2}}{2}\right)^{1/2}}$ \\
$T_{14C}(n),\;\epsilon=\sqrt{32}$  & $8\sqrt{2}$ & $\displaystyle{\frac{443}{392}-\frac{60\sqrt{2}}{49}}$ &  $\displaystyle{\frac{8}{\pi^{3/2}} \times \left(\frac{8\sqrt{2}-11}{7}\right)^{1/2}}$ \\
$T_{14\overline{C}}(n),\;\epsilon=-\sqrt{32}$  & $-(1+2\sqrt{2})^2$ 
& $\displaystyle{\frac{339}{784}+\frac{201\sqrt{2}}{196}}$ &  $\displaystyle{\frac{(1+2\sqrt{2})^2}{\pi^{3/2}} \times \left(\frac{\sqrt{14}}{7}+\frac{\sqrt{7}}{14}\right)}$ \\[1ex]
\hline 
$T_{15A}(n),\;\epsilon=1$ & $12$ & $\displaystyle{-\frac{489}{1000}}$ & $\displaystyle{\frac{6\sqrt{3}}{5\pi^{3/2}}}$   \\
$T_{15B}(n),\;\epsilon=-11$  & $-12$  & $\displaystyle{\frac{3}{8}}$ & $\displaystyle{\frac{12\sqrt{3}}{\pi^{3/2}}}$  \\
$T_{15C}(n),\;\epsilon=2i$  & $11+2i$ & $\displaystyle{-\frac{1209}{2000}+\frac{231}{1000}i}$ & $\displaystyle{\frac{(11+2i)^{3/2}}{20\pi^{3/2}}}$ \\
$T_{15\overline{C}}(n),\;\epsilon=-2i$ & $11-2i$ & $\displaystyle{-\frac{1209}{2000}-\frac{231}{1000}i}$ & $\displaystyle{\frac{(11-2i)^{3/2}}{20\pi^{3/2}}}$  \\[1ex]
\hline
$T_{24}(n)$ & $8$ & $\displaystyle{-\frac38}$  &  $\displaystyle{\frac{\sqrt{8}}{\pi^{3/2}}}$  \\[1ex]
\hline
\end{tabular}}}}}
\end{table}

%---------------------------------------------------

\section{Further conjectures: supercongruences}
Suppose $T(0)=1$ and that $\left\{T(n)\right\}$ satisfies a Lucas congruence for a prime~$p$; that is, 
for all positive integers $n$,
$$
T(n) \equiv T(n_0)T(n_1)\cdots T(n_r) \pmod p,
$$
where $n=n_0+n_1p+\cdots +n_rp^r$ is the $p$-adic expansion of $n$ in base $p$.
Then since $pn=n_0p+n_1p^2+\cdots +n_rp^{r+1}$  it follows immediately that the congruence
$$
T(pn) \equiv T(n) \pmod p
$$
holds for all positive integers~$n$. For the Ap{\'e}ry numbers it was shown by Gessel~\cite{gessel} that a stronger congruence holds, namely
for all primes $p\geq 5$ and all positive integers~$n$,
\begin{equation}
\label{A3r}
T(pn) \equiv T(n) \pmod {p^3},
\end{equation}
where $T(n) = A(n)$ are the Ap{\'e}ry numbers. Such a congruence, involving a power of the prime~$p$, is called
a supercongruence.
The Ap{\'e}ry numbers are the sequence labelled~6(A) in Table~\ref{table1} and it is natural to check if other sequences in Table~\ref{table5} satisfy this type of congruence.
Chan et al~\cite[Theorem 2.3]{ccs} proved that the 
congruence~\eqref{A3r} also holds for the sequence labelled~6(B) in Table~\ref{table1}.\footnote{The absolute
values of the numbers $s(n)$ labelled~6(B) in Table~\ref{table1} are called Domb numbers.} 
They conjectured~\cite[Conjecture 3.1]{ccs} that the congruence also holds for the Almkvist--Zudilin
numbers labelled~6(C) and this has been proved by T.~Amdeberhan and R.~Tauraso~\cite{amdeberhan}.
Further results about supercongruences satisfied by the sequences in Tables~\ref{table11} and~\ref{table1},
as well as the sequences in Table~\ref{1to35} corresponding to levels 7, 10 and~18, are
summarised by Straub in~\cite{straub23}.

Let us briefly discuss supercongruences for level~$7$. 
Let $T(n)$ be the level~7 sequence in Table~\ref{1to35} and let $t(n)$ be defined by
$$
\left(\sum_{n=0}^\infty t(n) X^n\right)^2 = \sum_{n=0}^\infty T(n)X^n
$$
so that
$$
\sum_{j=0}^n t(j)t(n-j) = T(n) = \sum_{j} {n \choose j}^{2} {2j \choose n} {n+j \choose j}.
$$
Chan et al~\cite[Conjecture~5.4]{ccs} conjectured that 
\begin{equation}
\label{c7a}
t(pn) \equiv t(n) \pmod{p^2} 
\end{equation}
for all positive integers $n$ and all primes $p$ that are quadratic residues modulo~$7$, and
L.~O'Brien~\cite[Conjecture 9.3]{OBrien} conjectured that the congruence also holds for~$p=7$.
As far as I know, this conjecture is still open.
For the other sequence it was conjectured~\cite[Conjecture 5.1]{sporadic} that
\begin{equation}
\label{c7b}
T(pn) \equiv T(n) \pmod{p^3}
\end{equation}
for all positive integers $n$ and all primes $p\geq 3$. The conjecture~\eqref{c7b} has been proved for primes $p\geq5$
by R.~Osburn, B.~Sahu and A.~Straub \cite[Theorem 1.2]{oss}.
Thus, the weight one sequence $\left\{ t(n)\right\}$ satisfies a mod $p^2$ congruence
for primes $p$ in half of the residue classes modulo 7, whereas the weight two sequence $\left\{T(n)\right\}$
satisfies the stronger mod $p^3$ congruence for all primes $p\geq 5$ (and conjecturally for $p=3$).

The discussion above prompts the question of whether supercongruences exist for any
of the self-starting sequences defined by four-term recurrence relations in Table~\ref{table5}.
Although mod $p^3$ congruences seem to be rare (I found some for the prime $p=2$ and one for the prime $p=3$),
I did find several curious mod $p^2$ congruences which are summarised in the following conjectures.

%---------------------------------------------------

\subsection{Level 11} Here is the conjecture for level 11.
\begin{conjecture}
\label{c11}
Let $\left\{T_{11}(n)\right\}$ be the level 11 sequence in Table~\ref{table5}. For all positive integers~$n$, we have
$$
T_{11}(pn) \equiv T_{11}(n) \pmod {p^2} \quad\text{if $p\in \left\{2,\,59,\,5581\right\}$}.
$$
Moreover, in the case $p=2$ we also have
$$
T_{11}(2n) \equiv T_{11}(n) \pmod {2^6} 
$$
unless $n=1$ or $n=1+2^{j-1}$ or $n=1+3\times 2^j$ for some $j\in\mathbb{Z}^+$.

\end{conjecture}
The primes 59 and 5581 in Conjecture~\ref{c11} are surprising. I have tested the primes~\mbox{$p<10000$} and did not find any others with this property, apart from~$p=2$.
There is not enough information to speculate about whether the mod $p^2$ congruences hold for any other primes~\mbox{$p>10000$.}
The conjectures have been tested for $pn\leq 10^6$, in other words for $1\leq n\leq 1694$ in the case~$p=59$, and
for $1\leq n \leq 17$ in the case $p=5581$.
While 17 data points in the latter case may not seem like much evidence, the~1694 data points is quite compelling evidence for the prime~$p=59$. 
Some statistics for other primes are given in~Table~\ref{t99}. The table suggests that the mod $p^2$ congruences might hold for particular values of~$n$ for other primes.
For example, we have:
\begin{conjecture}
$$
T_{11}(3n) \equiv T_{11}(n) + 3n \pmod {3^2}
$$
and
$$
T_{11}(5n) \equiv T_{11}(n)  \pmod {5^2} \quad\mbox{if and only if $\;5|n$}.
$$
\end{conjecture}
\noindent
This conjecture implies the values $c(3) = 333$ and $c(5) = 200$ in Table~\ref{t99}.
The reader should be warned, however, that the values $c(7) = 750$ and $c(11)=875$ should not to be taken as any suggestion that $3/4$ of the integers satisfy the mod~$p^2$ congruence for 
the primes $p=7$, or that the proportion is $7/8$ of the integers in the case $p=11$. It is pure coincidence that these proportions match with simple rational numbers,
and those proportions change as the range of~$n$ increases.

\begin{table}
\caption{\\
Values of 
$c(p) := \# \left\{ n : T_{11}(pn) \equiv T_{11}(n) \pmod {p^2} \;\text{and}\; 1\leq n\leq 1000 \right\}$ for primes $p\leq101$. Note the values of $c(2)$ and $c(59)$.} \label{t99a}
\label{t99}
{\renewcommand{\arraystretch}{2.4}
{\resizebox{12cm}{!}{
{\begin{tabular}{|c||c|c|c|c|c|c|c|c|c|c|c|c|c|}
\hline
$p$ & 2 & 3 & 5& 7 & 11 & 13 & 17 & 19 & 23 & 29 & 31 & 37 & 41  \\ \hline
$c(p)$ & 1000 & 333&200 &750&875&274&222&286&129&62&32&27&48 \\
\hline \hline
$p$ & 43 & 47 & 53 & 59 & 61& 67& 71& 73& 79& 83& 89& 97 &101 \\ \hline
$c(p)$ &87&87&18&1000&49 & 136& 56& 27& 63& 24& 11& 10 & 9\\
\hline
\end{tabular}}
}}}
\end{table}

It is worth comparing Conjecture~\ref{c11} with the corresponding level~11 conjecture of Chan et al:
\begin{conjecture}[Conjecture 5.5 in~\cite{ccs}]
Let $t_{11}(n)$ be defined by
$$
\left(\sum_{n=0}^\infty t_{11}(n) x^n\right)^2 = \sum_{n=0}^\infty T_{11}(n) x^n,
$$
where $T_{11}(n)$ is the level 11 sequence in Table~\ref{table5}. Equivalently,
$$
\sum_{j=0}^n t_{11}(j)t_{11}(n-j) = T_{11}(n) 
$$
or
\begin{align*}
(n+1)^2 t_{11}(n+1) &= 2(10n^2+5n+1)t_{11}(n) \\
&\qquad -8(7n^2-7n+2)t_{11}(n-1)+22(n-1)(2n-3)t_{11}(n-2)
\end{align*}
with $t_{11}(0)=1$ and $t_{11}(-1) = t_{11}(-2) = 0$.
Then
$$
t_{11}(pn) \equiv t_{11}(n) \pmod{p^2}
$$
for all positive integers $n$ and all primes $p$ with $\left(\frac{p}{11}\right) = 1$.
\end{conjecture}
In other words, the weight one sequence $\left\{t_{11}(n)\right\}$ is conjectured to satisfy a mod $p^2$ congruence
for  for the primes that are quadratic residues modulo 11. This is analogous to the level 7 conjecture given
by~\eqref{c7a}.
However, the
weight two sequence $\left\{T_{11}(n)\right\}$ appears to satisfy a mod $p^2$ congruence for a much smaller
and as yet undetermined set of primes. This is much weaker than the corresponding level~7
result in~\eqref{c7b}, both in terms of the strength of the congruence and also in terms of
the set of primes~$p$.

%---------------------------------------------------

\subsection{Level 14}
Here are the corresponding conjectures for level 14.
\begin{conjecture}
Let $\left\{T_{14A}(n)\right\}$, $\left\{T_{14B}(n)\right\}$, $\left\{T_{14C}(n)\right\}$ and $\left\{T_{14\overline{C}}(n)\right\}$ be the level~14 sequences in Table~\ref{table5}.
For all positive integers~$n$, we have
$$
T_{14A}(pn) \equiv T_{14A}(n) \pmod {p^2} \quad\text{if $p\in \left\{2,\,17,\,263\right\}$}
$$
and
$$
T_{14B}(pn) \equiv T_{14B}(n) \pmod {p^2} \quad\text{if $p\in \left\{2,\,7,\;17,\,263\right\}$}.
$$

Moreover, in the case $p=2$ we also have
$$
T_{14A}(2n) - T_{14A}(n) \equiv  \begin{cases}
0 \pmod {2^4} \quad\text{if $n\equiv 0$ or $3 \pmod 4$}, \\
8 \pmod {2^4} \quad\text{if $n\equiv 1$ or $2 \pmod 4$};
\end{cases}
$$
$$
T_{14B}(2n) - T_{14B}(n) \equiv  \begin{cases}
0\;\; \pmod {2^4} \quad\text{if $n\equiv 0 \pmod 2$}, \\
12 \pmod {2^4} \quad\text{if $n\equiv 1 \pmod 2$};
\end{cases}
$$
and
$$
T_{14C}(2n) \equiv T_{14C}(n) \pmod {2^6} \quad\text{and}\quad T_{14\overline{C}}(2n) \equiv T_{14\overline{C}}(n) \pmod {2^6} 
$$
unless $n=1,\,2,\,3$ or $n=2^j\times 3+1$ for some nonnegative integer~$j$.
\end{conjecture}
The conjecture has been tested for primes $p<10000$ and for values of $n$ satisfying $pn\leq 100000$.
The most obvious question is why the prime $p=7$ appears in the conjecture for $T_{14B}$ but not for $T_{14A}$.
As partial answer, we conjecture that
$$
T_{14A}(7n) \equiv T_{14A}(n) \pmod {7^2} \quad \text{if $n\equiv 0,\,4,\,5\;\text{or}\; 6\pmod{7}$}.
$$

%---------------------------------------------------

\subsection{Level 15}
Here are the corresponding conjectures for level 15.
Congruences for $T_{15\overline{C}}(n)$ can be deduced from the conjectures for $T_{15C}(n)$ complex conjugation and hence are not listed.

\begin{conjecture}
Let $\left\{T_{15A}(n)\right\}$, $\left\{T_{15B}(n)\right\}$ and $\left\{T_{15C}(n)\right\}$ be the level~15 sequences in Table~\ref{table5}.
Then congruences modulo $p^2$ exist for $p\in\left\{2,3,5\right\}$ as follows:
\begin{itemize}
\item
In the case $p=2$ we have
$$
T_{15A}(2n) - T_{15A}(n) \equiv  \begin{cases}
0 \pmod {2^3} \quad\text{if $n\equiv 0 \pmod 2$}, \\
4 \pmod {2^3} \quad\text{if $n\equiv 1 \pmod 2$};
\end{cases}
$$
$$
T_{15B}(2n) - T_{15B}(n) \equiv  \begin{cases}
16 \pmod {2^5} \quad\text{if $n\equiv 2 \pmod 4$}, \\
0\;\; \pmod {2^5} \quad\text{otherwise}, 
\end{cases}
$$
and
$$
T_{15C}(2n) \equiv T_{15C}(n) \pmod {2^3} 
$$
unless $n=1$ or $n=1+2^j$ for some nonnegative integer~$j$.
\item
In the case~$p=3$ we have
$$
T_{15A}(3n) \equiv T_{15A}(n) \pmod {3^3}
$$
unless $n=1$, $n=2$, $n=1+3^j$, or $n=1+2\times 3^j$ for some nonnegative integer~$j$.
And,
$$
T_{15B}(3n) \equiv T_{15B}(n) \pmod {3^3}
$$
unless $n=1$, $n=1+3^j$, or $n=1+2\times 3^j$ for some nonnegative integer~$j$.
Furthermore,
$$
T_{15C}(3n) - T_{15C}(n) \equiv  \begin{cases}
i\;\, \pmod {3} \quad\text{if $n\equiv 3,\,5\,\text{or}\;6 \pmod 8$}, \\
2i \pmod {3} \quad\text{if $n\equiv 1,\,2\,\text{or}\;7 \pmod 8$}, \\
0\; \pmod {3} \quad\text{otherwise}. 
\end{cases}
$$
\item
In the case~$p=5$ we have
$$
T_{15B}(5n) \equiv T_{15B}(n) \pmod {5^2}
$$
for all positive integers~$n$.
Furthermore, if $n-1 \neq 5^{j_1}+5^{j_2}+\cdots +5^{j_k}$
for distinct nonnegative integers $j_1,\ldots,j_k$, i.e., the base~5 expansion of $n-1$ does not consist entirely of 0's and 1's, then
$$
T_{15A}(5n) \equiv T_{15A}(n) \pmod {5^2}
$$
and
$$
T_{15C}(5n) \equiv T_{15C}(n) \pmod {5^2}.
$$
\end{itemize}

\end{conjecture}
Once again, this conjecture has been tested for primes $p<10000$ and for values of $n$ satisfying $pn\leq 100000$.

%---------------------------------------------------

\subsection{Level 24}
Finally, here are the corresponding conjectures for level 24.
\begin{conjecture}
Let $\left\{T_{24}(n)\right\}$ be the level~24 sequence in Table~\ref{table5}.
For all positive integers~$n$, we have
$$
T_{24}(pn) \equiv T_{24}(n) \pmod {p^2} \quad\text{if $p\in \left\{2,\,89,\,179,\,643\right\}$}.
$$
In the case $p=2$ we also have
$$
T_{24}(2n) \equiv T_{24}(n) \pmod {2^5} 
$$
unless $n=1$ or $n=1+2^{j}$ for some nonnegative integer~$j$.

\medskip
\noindent
In the case $p=3$ we have
$$
T_{24}(3n) - T_{24}(n) \equiv  \begin{cases}
3 \pmod {9} \quad\text{if $n\equiv 4\;\text{or}\;5 \pmod 6$}, \\
6 \pmod {9} \quad\text{if $n\equiv 1\;\text{or}\;2 \pmod 6$}, \\
0 \pmod {9} \quad\text{otherwise}. 
\end{cases}
$$
\end{conjecture}
The conjecture has been tested for primes $p<10000$ and for values of $n$ satisfying $pn\leq 100000$.

\bigskip
\noindent
{\bf Acknowledgement:} It is a pleasure to recall and acknowledge helpful discussions with Wadim Zudilin back when this work started in 2015.
I am grateful to the anonymous referee for suggestions that have improved the presentation of the work.

\bibliographystyle{amsalpha}

\end{document}